\documentclass[preprint,12pt]{elsarticle}
\usepackage[T1]{fontenc}
\usepackage[latin9]{inputenc}
\usepackage[normalem]{ulem}
\usepackage{amsthm}
\usepackage{amsmath}
\usepackage{esint, euscript}
\usepackage{color, graphicx, lscape}

\usepackage{amssymb}

\newcommand{\E}{{\mathbb E}}
\usepackage{soul}
\newcommand{\eqd}{\stackrel{d}{=}}
\def\I{I\!\!I}
\newcommand{\N}{\mathbb{N}}
\newcommand{\eps}{\varepsilon}
\newcommand{\floor}{\operatorname{floor}}
\newcommand{\wU}{\widetilde{U}_{n}}
\makeatletter

\theoremstyle{plain}
\newtheorem{theorem}{\protect\theoremname}
\newtheorem{lem}{\protect\lemmaname}
\newtheorem{rem}{\protect\remname}
\newtheorem{ex}{\protect\example}
\newtheorem{prop}{\protect\prop}
\newtheorem{cor}{\protect\corr}

\renewcommand{\i}{\textrm{i}}
\providecommand{\theoremname}{Theorem}
\providecommand{\remname}{Remark}
\providecommand{\lemmaname}{Lemma}
\providecommand{\example}{Example}
\providecommand{\prop}{Proposition}
\providecommand{\corr}{Corollary}
\providecommand{\defname}{Definition}

\renewcommand{\P}{\mathbb{P}}

\newcommand{\R}{{\mathbb R}}
\newcommand{\C}{{\mathbb C}}

\newcommand{\RR}{{\EuScript R}}

\newcommand{\Var}{\operatorname{Var}}
\newcommand{\M}{\mathcal{M}}

\newcommand{\F}{\mathcal{F}}
\newcommand{\K}{\mathcal{K}}
\newcommand{\G}{\mathcal{G}}
\renewcommand{\Re}{\operatorname{Re}}
\renewcommand{\Im}{\operatorname{Im}}
\newcommand{\mean}{\operatorname{mean}}

\journal{Stochastic processes and their applications}

\begin{document}

\begin{frontmatter}



\title{Low-frequency estimation of continuous-time moving average  L\'evy processes $^{1}$}


\author[1,2]{Denis Belomestny}
\author[2]{Vladimir Panov}
\author[3]{Jeannette H. C.  Woerner}
\address[1]{University of  Duisburg-Essen, Thea-Leymann-Str. 9, 45127 Essen,  Germany}
\address[2]{National Research University Higher School of Economics \\ Shabolovka, 26, 119049 Moscow,   Russia}
\address[3]{Technische Universit{\"a}t Dortmund, Vogelpothsweg 87, 44227 Dortmund, Germany}
\begin{abstract}
In this paper we study the problem of statistical inference for a continuous-time moving average  L\'evy process of the form 
\[
Z_{t}=\int_{\R}\mathcal{K}(t-s)\, dL_{s},\quad t\in\mathbb{R}
\]
 with a deterministic kernel \(\K\) and a  L{\'e}vy process \(L\). Especially the  estimation of  the L\'evy measure \(\nu\) of  $L$ from low-frequency observations of the process $Z$ is considered.
We construct a consistent estimator, derive its convergence rates and  illustrate its performance  by a numerical example.  On the technical level, the main challenge  is to establish  a kind of exponential mixing  for  continuous-time moving average  L\'evy processes.  
\end{abstract}

\end{frontmatter}
\footnotetext[1]{The study has been funded by the Russian Academic Excellence Project ``5-100''. The first and the third author  acknowledge the financial support from  the Deutsche
Forschungsgemeinschaft through the SFB 823 ``Statistical modelling of nonlinear dynamic processes''. \\ \newline
E-mail addresses: denis.belomestny@uni-due.de (D.Belomestny),  \\
vpanov@hse.ru (V.Panov), jwoerner@mathematik.uni-dortmund.de (J.Woerner)}
\section{Introduction}
\label{sec: intr}
Continuous-time L\'{e}vy-driven moving average  processes of the form: \[Z_t=\int_{-\infty}^\infty \mathcal{K}(s,t)\, dL_s\] with  a deterministic kernel $\mathcal{K}$ and a L\'{e}vy process $(L_t)_{t\in \mathbb{R}}$  build a large class of stochastic processes including semimartingales and non-semimartingales, cf. Basse and Pedersen \cite{BassePedersen2009}, Basse-O'Connor and Rosinsky \cite{BasseRosinski2016},  Bender, Lindner and Schicks \cite{BenderLindnerSchicks2012}, as well as long-memory processes. Starting point was the paper by Rajput and Rosinski \cite{rajput1989spectral} providing conditions on the interplay between $\mathcal{K}$ and $L$ such that $Z$ is well defined. Continuous-time L\'{e}vy-driven moving average processes provide a unifying approach to many popular stochastic models like L\'{e}vy driven Ornstein-Uhlenbeck processes, fractional L\'{e}vy processes and CARMA processes. Furthermore, they are the building blocks of more involved models such as L\'{e}vy semistationary processes and ambit processes, which are popular in turbulence and finance, cf. Barndorff-Nielsen, Benth and Veraart \cite{Barndorff-NielsenBenthVeraart2015}.  
\par
Statistical inference for Ornstein-Uhlenbeck processes and CARMA processes is already well-established due to the special structure of the processes, for an overview see Brockwell and Lindner \cite{BrockwellLindner2012}, whereas for general continuous-time L\'{e}vy driven moving average  processes so far only partial results are known in the literature mainly concerning parameters which enter the kernel function, cf. Cohen and Lindner \cite{CohenLindner2013} for an approach via empirical moments or Zhang, Lin and Zhang \cite{ZhangLinZhang2015} for a least squares approach. Further results concern limit theorems for the power variation, cf.  Glaser \cite{Glaser2015}, 
Basse-O'Connor, Lachieze-Rey and Podolskij \cite{BasseLachiezePodolskij2015}, which may be used for statistical inference based on high-frequency data.
\par
In this paper we consider a special case of stationary continuous-time L\'{e}vy-driven moving average  processes of the form $Z_t=\int_{-\infty}^\infty \K(s-t)\, dL_s$ and aim to infer the unknown parameters of the driving L\'{e}vy process from its low-frequency observations. Our setting especially includes the case of Gamma-kernels of the form $\K(t)=t^\alpha e^{-\lambda t}1_{[0,\infty)}(t)$ with $\lambda>0$ and $\alpha>-1/2$, which serve as a popular kernel for applications in finance and turbulence, cf. Barndorff-Nielsen and Schmiegel  \cite{Barndorff-NielsenSchmiegel2009}. The special symmetric
case of the well-balanced Ornstein-Uhlenbeck process has been discussed in
Schnurr and Woerner \cite{WD}.
\par
 In fact, the resulting statistical problem is rather challenging for several reasons. On the one hand, the set of parameters, i.e., the so-called L\'evy-triplet of the driving L\'evy process contains, in general, an infinite dimensional object, a L\'evy measure making the statistical problem nonparametric. On the other hand, the relation between the parameters of the underlying L\'evy process \((L_t)\) and those of the resulting moving average process \((Z_t)\) is rather nonlinear and implicit, pointing out to a nonlinear ill-posed statistical problem. It turns out that in Fourier domain this relation becomes exponentially linear and has a form of  multiplicative convolution. This observation underlies our estimation procedure, which basically consists of three steps. First, we estimate the marginal characteristic function of the L\'{e}vy-driven moving average  process \((Z_t)\). Then we estimate the Mellin transform of the second derivative of the log-transform of the characteristic function. Finally, an inverse Mellin transform technique is used to reconstruct the L\'evy density of the underlying L\'evy process.  
 \par
 The paper is organized as follows. In the next session, we explain our setup and discuss the correctness of our model. In Section~\ref{melsec}, we present the estimation procedure. Our main theoretical results related to the rates of convergence of the estimates are given in Section~\ref{conv}. Next, in Section~\ref{num}, we provide a numerical example, which shows the performance of our procedure. All proofs are collected in the appendix.
\section{Setup}
\label{sec: setup}
In this paper we study a stationary continuous-time moving average (MA) L\'evy process $\left(Z_{t}\right)_{t\in\mathbb{R}}$
of the form:  
\begin{equation}
Z_{t}=\int_{-\infty}^{\infty}\mathcal{K}(t-s)\, dL_{s},\quad t\in\mathbb{R},\label{eq:xt}
\end{equation}
where $\mathcal{K}:\mathbb{R}\to\mathbb{R}_{+}$ is a  measurable function and
$\left(L_{t}\right)_{t\in\mathbb{R}}$ is a two-sided L\'evy process
with the triplet $\mathcal{T}=(\gamma,\sigma^2,\nu)$.   As shown in \cite{rajput1989spectral}, under the conditions 
\begin{eqnarray}
\label{cond1}
\int_{\mathbb{R}}\int_{\mathbb{R}\setminus\{0\}}\left(\left|\mathcal{K}(s)x\right|^{2}\wedge 1\right)\nu(dx)\, ds<\infty,
\\
\label{cond2}
\sigma^2\int_{\mathbb{R}}\mathcal{K}^2(s)\,ds<\infty,
\\
\label{cond3}
\int_{\mathbb{R}}\left|\mathcal{K}(s)\left(\gamma+\int_{\mathbb{R}}x\left(1_{\{\left|x\mathcal{K}(s)\right|\leq1\}}-1_{\{\left|x\right|\leq1\}}\right)\nu(dx)\right)\right|\,ds<\infty 
\end{eqnarray}
 the stochastic integral in (\ref{eq:xt}) exists.  In what follows, we assume that \(\mathcal{K} \in L^{1}(\R) \cap L^{2}(\R)\)  and the L{\'e}vy measure \(\nu\) satisfies
  \begin{eqnarray}
\label{sec_mom}
 \int x^2\nu(dx)<\infty,
\end{eqnarray}
that is, the L\'evy process \(L\) has finite second moment.  In fact, (\ref{cond2}) is trivial in this case; 
 condition
 (\ref{cond1}) directly follows from the inequality
 \begin{eqnarray*}
\int_{\mathbb{R}}\int_{\mathbb{R}\setminus\{0\}}\left(\left|\mathcal{K}(s)x\right|^{2}\wedge 1\right)\nu(dx)\, ds 
&\leq&
\int_{\mathbb{R}}\int_{\mathbb{R}\setminus\{0\}}\left|\mathcal{K}(s)x\right|^{2}\nu(dx)\, ds \\
&=&
\int_{\mathbb{R}}\left(\mathcal{K}(s)\right)^{2}
ds 
\cdot 
\int_{\mathbb{R}\setminus\{0\}}x^{2}\nu(dx)\, ds.
\end{eqnarray*}
As to the condition  (\ref{cond3}), we have
\begin{multline*}
\int_{\mathbb{R}}\left|\mathcal{K}(s)\left(
\gamma
- \int_{\mathbb{R}} x 1_{\{\left|x\right|\leq1\}}  \nu(dx)
\right) 
+\int_{\mathbb{R}}x \mathcal{K}(s)
1_{\{\left|x\mathcal{K}(s)\right|\leq1\}}
\nu(dx)
\right|\,ds
\\
=
\int_{\mathbb{R}}\left|\mathcal{K}(s)
\E \left[ L_{1} \right]
- \int_{\mathbb{R}}x \mathcal{K}(s)
1_{\{\left|x\mathcal{K}(s)\right|>1\}}
\nu(dx)
\right|\,ds
\\
\leq 
\left| \E[L_{1}] \right|\int_{\R} \mathcal{K}(s)\, ds
+ 
\int_{\R}\int_{\R} x^{2}  \left(\mathcal{K}(s)\right)^{2} \nu(dx)\,ds.
\end{multline*}
In the sequel we assume that \(\mathcal{K} \in L^{1}(\R) \cap L^{2}(\R)\) and
\begin{eqnarray}
\label{sec_mom}
 \int x^2\nu(dx)<\infty.
\end{eqnarray}
Moreover, under the above assumptions, the process
$\left(Z_{t}\right)_{t\in\mathbb{R}}$ is strictly stationary with the characteristic function of the form
\begin{eqnarray}
\label{Four}
\Phi(u)  :=  \E\left[e^{\i uZ_{t}}\right] =\exp\left(\psi(u)\right),
\end{eqnarray}
 where 
 \begin{eqnarray*}
\Psi(u) := \int_{\R}\psi(u\,\mathcal{K}(s))\, ds
\end{eqnarray*}
and
\[
\psi(u):= \i u\gamma -\sigma^2u^2/2+\int_{\mathbb{R}}\left(e^{\i ux}-1-\i ux 1_{\{|x|\leq 1\}}\right)\nu(dx).
\]
Our main goal is the estimation of the parameters of the L\'evy process
$L$ from low-frequency observations of the process $Z$ given that
the function $\mathcal{K}$ is known.
\section{Mellin transform approach}
\label{melsec}
\subsection{Main idea}
\label{Mell}
Let \(L\)  be a L\'evy process  with the
L{\'e}vy triplet \((\mu,\sigma^2, \nu),\) where \(\nu\) is  an absolutely continuous w.r.t. to the Lebesgue measure on \(\mathbb{R}_{+}\) and satisfies \eqref{sec_mom}.  Denote by \(\nu(x)\) the density of \(\nu\) and set  \(\overline{\nu}(x):= x^2\nu(x).\)  For the sake of clarity we first assume that \(\sigma\) is known and  \(\nu\) is supported on \(\R_+,\) i.e. \(L\) is a sum of a Brownian motion and subordinator.  Set 
\[
\Psi_{\sigma}(u):=\Psi(u)+\frac{\sigma^2u^2}{2}\int_{\R} 
	\mathcal{K}^2(x)\, dx. 
\] 
It follows then 
\begin{eqnarray*}
	\Psi_{\sigma}''(u) 
&=&
	\int_{\R} \psi''\left(
		u \mathcal{K}(x) 
	\right)
	\cdot\mathcal{K}^2(x)\,
	dx
	=
	-\int_{\R} \mathcal{F}[\overline{\nu}]\left(
		u \mathcal{K}(x)
	\right) \cdot
	\mathcal{K}^2(x)\, dx,
\end{eqnarray*}
{where \(\mathcal{F}[\overline{\nu}]\) stands for the Fourier transform of \(\overline{\nu}.\)}
Next, let us compute the Mellin transform of \(\Psi_{\sigma}''\): 
\begin{eqnarray}
\nonumber
\M\left[
	\Psi_{\sigma}''
\right]
(z) &=& 
-
\int_{\R_{+}}
\left[
\int_{\R} \mathcal{F}[\overline{\nu}]\left(
		u \mathcal{K}(x)
	\right) \cdot
	\mathcal{K}^2(x)\, dx
\right]
	u^{z-1}
	du\\
	\nonumber
&=&
-
\int_{\R}
\left[
\int_{\R_{+}} \mathcal{F}[\overline{\nu}]\left(
		u \mathcal{K}(x)\,
	\right) \cdot
		u^{z-1}
du
\right]
		\mathcal{K}^{2}(x)\, dx\\
&=&
-\mathcal{M}[\mathcal{F}[\overline{\nu}]](z)
\cdot
\left[
\int_{\R}
\left(
	\mathcal{K}(x)	
\right)^{2-z}	 dx
\right],
		\label{ML}
\end{eqnarray}
for all \(z\) such that\(\int_{\mathbb{R}} (\mathcal{K}(x))^{{2-\Re (z)}}\, dx<\infty\) and \(\int_{\R_{+}} \left |\mathcal{F}[\overline{\nu}]\left(
		v
	\right)\right| \cdot
		v^{\Re(z)-1}
dv<\infty.\) 
Since 
$\overline{\nu} \in L_{1}(\mathbb{R}_{+})$, it holds 
\begin{eqnarray*}
\mathcal{M}[\mathcal{F}[\overline{\nu}]](z) & = & \int_{0}^{\infty}v^{z-1}\left[\int_{0}^{\infty}e^{\i xv}\overline{\nu}(x) dx\right]\, dv\\
 & = & \M \bigl[
e^{\i \cdot}  
\bigr](z) 
\cdot
\mathcal{M}\bigl[\overline{\nu}\bigr](1-z).
\end{eqnarray*}
Note that the Mellin transform \(\mathcal{M}\bigl[\overline{\nu}\bigr](1-z)\) is defined for all \(z\) with \(\Re(z)\in (0,1),\) provided \(\overline{\nu}\) is bounded at \(0.\) Next, using the fact that 
\[
\M [e^{\i \cdot}](z) =  \Gamma(z)\left[\cos(\pi z/2)+\i\sin(\pi z/2)\right] = \Gamma(z) e^{\i \pi z / 2}
\]
for all \(z\) with  $ \mathrm{Re}(z) \in (0,1)$ (see \cite{Ober},  5.1-5.2), we get  
\[
\mathcal{M}[\Psi_{\sigma}''](z)= Q(z) \cdot\mathcal{M}[\overline{\nu}](1-z), \quad \Re(z)\in (0,1),
\]
where 
\begin{eqnarray}
\label{eq: Q}
Q(z):= -\Gamma(z)e^{\i \pi z / 2}\int_{\R}
\left(
	\mathcal{K}(x)	
\right)^{{2-z}}	 dx.
\end{eqnarray}

Finally, we apply the inverse Mellin transform to get 
\begin{eqnarray}
\nonumber
\overline{\nu}(x)  & = & 
 \frac{1}{2\pi \i}\int_{c-\i \infty}^{c+\i \infty}\, \mathcal{M}[\overline{\nu}](z)  x^{-z} dz\\
 \label{eq: main_rel}
 & = &
 \frac{1}{2\pi \i}\int_{c-\i \infty}^{c+\i \infty}\, 
 \frac{\mathcal{M}[\Psi_{\sigma}''](1-z)} 
 {Q(1-z)}  x^{-z} dz
\end{eqnarray}
for \(c\in (0,1).\) The formula \eqref{eq: main_rel} connects the weighted Levy density \(\overline{\nu}\) to the characteristic exponent \(\Psi_{\sigma}\) of the process \(Z\) and forms the basis for our estimation procedure.
{
\begin{rem}
If \(\sigma^2\) is supposed to be unknown, one can estimate it by noting that 
for a properly chosen bounded kernel \(w\) with \(\mathrm{supp}(w)\subseteq [1,2]\) and \(\int_0^\infty w(u) \,du=1, \)
\begin{eqnarray*}
\int_{\R_+} w_n(u) \Psi'' (u)\, du&=&-\sigma^2 \int_{\R} 
	\mathcal{K}^2(x)\, dx\\
	&&\hspace{2cm}-\int_{\R}\int_{\R_+} w_n(u) \mathcal{F}[\overline{\nu}]\left(
		u \mathcal{K}(x)
	\right) \mathcal{K}^2(x) \, du\,dx
\\
&=&-\sigma^2 \int_{\R} 
	\mathcal{K}^2(x)\, dx\\
	&&\hspace{2cm}-\int_{\R}\int_{\R_+} w(u) \mathcal{F}[\overline{\nu}]\left(
		u U_n \mathcal{K}(x)
	\right) \mathcal{K}^2(x) \, du\,dx	
\end{eqnarray*}
with \(w_n(u):=U_n^{-1}w(u/U_n)\) and some sequence \(U_n\to \infty.\)
Suppose that \(|\mathcal{F}[\overline{\nu}](u)|\leq C(1+u)^{-\alpha}\)  for all \(u\geq 0\) and some constants \(\alpha>0,\) \(C>0,\) then 
\begin{eqnarray*}
\left|\int_{\R}\int_{\R_+} w(u) \mathcal{F}[\overline{\nu}]\left(
		u U_n \mathcal{K}(x)
	\right) \mathcal{K}^2(x) \, du\,dx\right |\leq \|w\|_\infty\int_{\R} \frac{\mathcal{K}^2(x)}{(1+U_n \mathcal{K}(x))^{\alpha}}
	\, dx\to 0
\end{eqnarray*}
as \(n\to \infty.\)   For example, in the case of a one-sided exponential kernel \(\K(x)=e^{-x}\I(x\geq 0),\) we derive
\begin{eqnarray*}
\int_{\R} \frac{\mathcal{K}^2(x)}{(1+U_n \mathcal{K}(x))^{\alpha}}
	\, dx=\frac{1}{U_{n}^{-2}}\int_{0}^{U_{n}}\frac{z}{(1+z)^{\alpha}}\,dz\lesssim 
	\begin{cases}
	U_n^{-\alpha}, & \alpha<2, \\
	U_n^{-2}\log (U_n) , & \alpha=2, \\
	U_n^{-2}, & \alpha>2,
	\end{cases}
\end{eqnarray*}
as \(n\to \infty.\)
\end{rem}
}
\begin{rem}
Let us remark on the general case where the jump part of \(L\) is not necessary a subordinator. In this case one can show that
\begin{multline*}
\frac{\mathcal{M}\left[\Psi_{\sigma}''(-\cdot)\right](u)+\mathcal{M}\left[\Psi_{\sigma}''(\cdot)\right](u)}{2}=
\\
{-} \left\{ \mathcal{M}\bigl[\overline{\nu}_{+}\bigr](1-z)+\mathcal{M}\bigl[\overline{\nu}_{-}\bigr](1-z)\right\} 
\\
{ \cdot\cos\left(
\frac{\pi z}{2}
\right)
\Gamma(z)
\cdot \int_{\R} \left(
\K(x)
\right)^{2-z} dx
}
\end{multline*}
and
\begin{multline*}
\frac{\mathcal{M}\left[\Psi_{\sigma}''(\cdot)\right](u)-\mathcal{M}\left[\Psi_{\sigma}''(-\cdot)\right](u)}{2\i}=
\\
{-}
\left\{ \mathcal{M}\bigl[\overline{\nu}_{+}\bigr](1-z)-\mathcal{M}\bigl[\overline{\nu}_{-}\bigr](1-z)\right\}
\\
{ \cdot\sin\left(
\frac{\pi z}{2}
\right)
\Gamma(z)
\cdot \int_{\R} \left(
\K(x)
\right)^{2-z} dx
},  
\end{multline*} 
where \(\overline{\nu}_{+}(x)=\nu(x)\cdot 1(x\geq 0)\) and \(\overline{\nu}_{-}(x)=\nu(-x)\cdot 1(x\geq 0).\) Using the above formulas, one can express \(\mathcal{M}\bigl[\overline{\nu}_{-}\bigr],\)    \(\mathcal{M}\bigl[\overline{\nu}_{+}\bigr]\)  in terms of \(\mathcal{M}\left[\Psi_{\sigma}''(-\cdot)\right],\) \(\mathcal{M}\bigl[\overline{\nu}_{-}\bigr]\) and apply the  Mellin inversion formula to reconstruct \(\overline{\nu}_{-}\) and \(\overline{\nu}_{+}.\)
 \end{rem}
\subsection{Estimation procedure}
Assume that the process \(Z\) is observed on the equidistant time grid \( \{\Delta,  2 \Delta, \ldots ,\) \( n\Delta\}.\) Our aim is to estimate the L{\'e}vy density \(\nu\) of the process \(L.\)
First we approximate the Mellin transform of the function\[\Psi_{\sigma}''(u) =
\frac{\Phi''(u)}{\Phi(u)}-\left(
	\frac{\Phi'(u)}{\Phi(u)}
\right)^{2}+\sigma^2\|\K\|_{L^2}^2
\]
 via
\begin{equation}
\mathcal{M}_{n}[\Psi_{\sigma}''](1-z):=\int_{0}^{U_{n}}\left[\frac{\Phi''_{n}(u)}{\Phi_{n}(u)}-\left(\frac{\Phi'{}_{n}(u)}{\Phi_{n}(u)}\right)^{2}+\sigma^2\|\K\|_{L^2}^2\right]u^{-z}\,du, \label{eq:emp_mel_psi}
\end{equation}
where 
\[
\Phi_{n}(u) := \frac{1}{n} \sum_{k=1}^{n} 
	\exp\{\i Z_{k \Delta} u\}
\]
and a sequence $U_{n}\to\infty$ as $n\to\infty.$ 
Second, by regularising the inverse Mellin transform, we define 
\begin{eqnarray}
\label{nunu}
\overline{\nu}_{n}(x):=  \frac{1}{2\pi \i}\int_{c-\i V_{n}}^{c+\i V_{n}}\,  \frac{\mathcal{M}_{n}[\Psi_{\sigma}''](1-z)} 
 {Q(1-z)}  x^{-z} dz
\end{eqnarray}
for some  \(c \in (0,1)\)  and some sequence $V_{n}\to\infty,$ which will be specified later.  In the next section we study the properties of the estimate  \(\overline{\nu}_{n}(x).\) In particular, we show that  \(\overline{\nu}_{n}(x)\) converges to \(\overline{\nu}(x)\) and derive the corresponding convergence rates.
\section{Convergence}
\label{conv}

Assume that the following conditions hold.
\begin{enumerate}
\item[(AN)] For some \(A>0\) and \(\alpha\in (0,1), {\gamma>0},  {c \in (0,1)}\)  the L\'evy density \(\nu\) fulfills
\begin{eqnarray}\label{assf}
\int_{\R} \left(
	1+|y|
\right)^{\alpha} 
\left| 
	\F[\overline{\nu}] (y)
\right|
dy \leq  A, 
\end{eqnarray}
\begin{eqnarray}
\label{assm}
\int_{\R} e^{\gamma |u|}\left|\mathcal{M}[\overline{\nu}]({c}+\i u)\right| du \leq A,
\end{eqnarray}
\begin{eqnarray}
\label{asbg}
\int_{\R_+}  (x\vee x^2)\, \nu(x)\,dx \leq A.
\end{eqnarray}
\end{enumerate}

\begin{theorem} \label{thm1}
Suppose that (AN) holds, \(\mathcal{K}\) is a nonnegative kernel with \(\mathcal{K}\in  L^1(\mathbb{R})\cap  L^2(\mathbb{R})\).
Denote $D_{j}(u):= ( \Phi_{n}^{(j)}(u)-\Phi^{(j)}(u)) / {\Phi(u)},$ $j=0,1,2,\ldots $ Let  for any real valued function \(f\) on \(\mathbb{R},\)  \(\left\|f\right\|_{U_n}:=\sup_{u\in[-U_{n},U_{n}]}\left|f(u)\right|.\) Fix some \(K>0\) and denote 
\begin{eqnarray*}
\mathcal{A}_K:=\left\{\max_{j=0,1,2}\left\|D_j\right\|_{U_n}\geq K \varepsilon_n\right\}, \qquad K \geq 0.
\end{eqnarray*}
Let \(\eps_{n}, U_{n}\) be two sequences of positive numbers such that \(\eps_{n}\to 0, U_{n} \to \infty\) as \(n \to \infty\), and moreover
\[
{K \varepsilon_n}\left(1+\left\| \Psi_{\sigma}'\right\|_{U_n}\right)\leq 1/2.
\]
Choosing \(\varepsilon_n\) and \(U_n\)  in such a way is always possible, since  \(\Psi_{\sigma}'(0)=\psi'(0)\int\mathcal{K}(s)\,ds\) is finite. 
Then on the set \(\overline{\mathcal{A}}_K\) the estimate \(\overline{\nu}_n(x)\) given by    \eqref{nunu} with the same \(c\in (0,1)\) as in \eqref{assm} satisfies
\begin{eqnarray}\nonumber
\sup_{x\in \mathbb{R}_+} \left\{x^{c}\left |\overline{\nu}_{n}(x) - \overline{\nu}(x)\right|\right\} & \leq  & \frac{1}{2\pi}\int_{\{|v|\leq V_{n}\}}\frac{\Omega_n}{|Q(1-c-\i v)|}\, dv+\frac{A}{2\pi}e^{-\gamma V_{n}},
\end{eqnarray}
{where \(Q\) as in \eqref{eq: Q}, \(V_n\) is a sequence of positive numbers  and}
\begin{multline*}
\Omega_n=2K\varepsilon_nU_n^{1-c} \left(2+\left\|\Psi_{\sigma}''\right\|_{U_n}+\left\|\Psi_{\sigma}'\right\|^2_{U_n}+3\left\|\Psi_{\sigma}'\right\|_{U_n}\right)
\\
+\left(A+\frac{2^{\alpha} A}{1-c}\right)\int_{\R}
\bigl[
\K(x)
\bigr]^{c+1}
\bigl[1+U_{n}\K(x)\bigr]^{-\alpha}\,
dx.
\end{multline*}
\end{theorem}
\begin{rem}
Note that in case of \(\mathrm{supp}(\nu)\subseteq \R_+\), the sum 
 \(2+\left\|\Psi_{\sigma}''\right\|_{U_n}+\left\|\Psi_{\sigma}'\right\|^2_{U_n}+3\left\|\Psi_{\sigma}'\right\|_{U_n}\) can be  uniformly bounded. Indeed,
\begin{eqnarray*}
\bigl| 
	\psi'(u)-\sigma^2 u 
\bigr| &=& 
\bigl| 
	\i \mu + \int_{\R_{+}} \i x e^{\i ux} \nu(x) dx
\bigr| \leq \mu + \int_{\R_{+}} x \nu(x) dx \leq \mu + A,
\end{eqnarray*}
 by \eqref{asbg}. Analogously,
\begin{eqnarray*}
\bigl| 
	\psi''(u)-\sigma^2 
\bigr| &=& 
\left| 
	\int_{\R_{+}}  x^{2} e^{\i ux} \nu(x) dx
\right| \leq  \int_{\R_{+}} x^{2} \nu(x) dx \leq A.
\end{eqnarray*}
Therefore 
\begin{eqnarray*}
\bigl\|
	\Psi_{\sigma}'
\bigr\|_{U_{n}}
&=&
\left\|
	\int_{\R} (\psi' (u \K(x))-\sigma^2u) \K(x) dx
\right\|_{U_{n}}
\leq 
\left(
	\mu + A
\right)
\|\K\|_{L^1},\\
\bigl\|
	\Psi_{\sigma}''
\bigr\|_{U_{n}}
&=&
\left\|
	\int_{\R} (\psi'' (u \K(x))-\sigma^2) \K^{2}(x) dx
\right\|_{U_{n}}
\leq 
A
\|\K\|_{L^2}^2,
\end{eqnarray*}
where the integrals in the right-hand sides are bounded due to the assumption \(\mathcal{K}\in L^1(\mathbb{R})\cap  L^2(\mathbb{R})\).
\end{rem}
\begin{ex}
Consider a tempered stable L\'evy process \((L_t)\) with 
\begin{align}
\label{gamma}\nu(x)=x^{-\eta-1}\cdot e^{-\lambda x}%
, \quad x\geq0, \quad \eta\in (0,1),\quad \lambda>0.
\end{align}
Since
\begin{align*}
\mathcal{M}[\overline{\nu}](z)=\lambda^{\eta-z-1}\Gamma(z-\eta+1), \quad\mathsf{Re}(z)>\eta-1,
\end{align*}
we derive that \eqref{assm} holds  for all   $0<\gamma<\pi/2$ and $\alpha>0$  due to the asymptotic
properties of the Gamma function. Furthermore, 
\begin{eqnarray*}
\F[\overline{\nu}] (u)=(\i u-\lambda)^{-(2-\eta)}\Gamma(2-\eta)
\end{eqnarray*}
and hence \eqref{assf} holds for any \(0<\alpha<2-\eta.\) 
Moreover, \(\nu\) satisfies  \eqref{asbg}.
\end{ex}

Let us now estimate the probability of the event \(\mathcal{A}_K.\) The following result holds. 
\begin{theorem}
\label{ak_bound}
Suppose that the following assumptions are fulfilled.
\begin{enumerate}
\item  The kernel \(\mathcal{K}\) satisfies 
\begin{eqnarray}
\label{assfourierk}
\sum_{j=-\infty}^{\infty}\left|\mathcal{F}[\mathcal{K}]\left(2\pi\frac{j}{\Delta}\right)\right|\leq K^{*}
\end{eqnarray}
and
\begin{eqnarray}
\label{assdecayk}
(\mathcal{K}\star\mathcal{K})(\Delta j)
\leq \kappa_{0} \; |j|^{\kappa_{1}} e^{-\kappa_{2}|j|},  \quad \forall \; j\in \mathbb{Z}
\end{eqnarray}
for some positive constants \(K^*,\)\(\kappa_0,\kappa_1\) and \(\kappa_2,\) such that the all eigenvalues of the matrix   \( ((\mathcal{K}\star\mathcal{K})(\Delta (j-k))_{k,j\in \mathbb{Z}}\) are bounded from below and above by two finite positive constants.
\item The L{\'e}vy measure \(\nu\) satisfies 
\[
\int_{\left|x\right|>1}e^{Rx}\nu(dx)\leq A_{R}
\]
for some  \(R>0\)  and \(A_R>0.\) 
\end{enumerate}
 Then under the choice 
 \[\varepsilon_n=\sqrt{\frac{\log(n)}{ n}}\cdot \exp\left(C_{1} \sigma^2 U_n^2\, \int (\mathcal{K}(x))^2\,dx\right)
\]
with \(C_{1}=A/2,\)
it holds 
 for  any \(K>0\)
\begin{eqnarray*}
\P\Bigl(\mathcal{A}_K\Bigr)\leq  
\frac{
C_{2}
}{
	\sqrt{K}
}
\frac{
	\sqrt{U_{n}} n^{(1/4) - C_{3}K^{2}}
}
{
	\log^{1/4}(n)
}
,
\end{eqnarray*}
where the positive constants \(C_{1}, C_{2}\) may depend on \(K^*,\) \(A_R\) and \(\kappa_{i},\) \(i=1,2.\) {Hence by an appropriate choice of $K$ we can ensure that  $\P({\cal A}_K)\to 0$ as \(n\to \infty\).}
\end{theorem}
\begin{ex}\label{cor1}
Consider the class of symmetric  kernels of the form
\begin{eqnarray}
\label{eq: xex_kernel}
\mathcal{K}(x) = |x|^{r} e^{-\rho|x|},
\end{eqnarray}
where \(r\) is a nonnegative integer and \(\rho>0.\) Let us check the assumptions of Theorem~\ref{ak_bound}.
We have 
\begin{eqnarray*}
\mathcal{F}[\mathcal{K}](u)=\Gamma(r+1)\left[\frac{1}{\left(\i u-\rho\right)^{r+1}}+\frac{1}{\left(-\i u-\rho\right)^{r+1}}\right]
\end{eqnarray*}
and \eqref{assfourierk} holds. Assumption~\eqref{assdecayk} is proved in Lemma~\ref{xre}. 
\end{ex}
\begin{cor}
\label{cor_xre}
Consider again a class of kernels of the form
\begin{eqnarray*}
\mathcal{K}(x) = |x|^{r} e^{-\rho|x|},
\end{eqnarray*}
where \(r\) is a nonnegative integer and \(\rho>0,\) {and assume that the L{\'e}vy measure \(\nu\) satisfies the set of assumptions (AN).} Then 
\begin{eqnarray*}
\Omega_n\lesssim K\varepsilon_n U_n^{1-c}+U_n^{-\alpha},\quad n\to \infty
\end{eqnarray*}
and 
\begin{eqnarray*}
\int_{\{|v|\leq V_{n}\}}\frac{1}{|Q(1-c-\i v)|}\, dv \lesssim 
\begin{cases}
V_{n}^{c+3/2}, & r=0,
\\
V_{n}^{c + 1}, & r\geq 1. 
\end{cases}
\end{eqnarray*}
As a result we have on \(\overline{\mathcal{A}}_K\) 
\begin{eqnarray*}
\sup_{x\in \mathbb{R}_+} \left\{x^{c}\left |\overline{\nu}_{n}(x) - \overline{\nu}(x)\right|\right\} 
\lesssim V_{n}^{\zeta }\left(
\varepsilon_n U_{n}^{(1-c)}+ U_{n}^{-\alpha}
\right)
+e^{-\gamma V_{n}}
\end{eqnarray*}
with \(\zeta=c+1+\I\{r=0\}/2\). By taking \(V_n=\varkappa \log (U_n)\) with \(\varkappa>\alpha/\gamma \) and  \(U_n={\theta}\log^{1/2} (n)\) for any \({\theta< \left(  A \int (\mathcal{K}(x))^2\,dx \right)^{-1/2}},\)
\begin{eqnarray*}
\sup_{x\in \mathbb{R}_+} \left\{x^{c}\left |\overline{\nu}_{n}(x) - \overline{\nu}(x)\right|\right\} 
\lesssim \log^{-\alpha/2}(n),\quad n\to \infty.
\end{eqnarray*}
\end{cor}

\subsection{Discussion}
The proof of Theorem~\ref{ak_bound} is based on some kind of exponential mixing  for  the general L\'evy-driven moving average processes of the form \eqref{eq:xt}. In fact, such mixing properties were established in the literature only for the processes \(Z\) corresponding to the exponential kernel function \(\mathcal{K},\) see, e.g. \cite{ilhe2015nonparametric}. The assumption of Theorem~\ref{ak_bound} may seem to be strong, but as shown above, are fulfilled for the family of kernels \eqref{eq: xex_kernel}.
\section{Numerical example}
\label{num}
\subsection{Simulation.}  Consider the integral \(
Z_{t}:=\int_{\mathbb{R}}\K(t-s)\, dL_{s}
\) with the kernel \(\K(x) = e^{-|x| } \) and  the L{\'e}vy process
\[L_{t}=L_{t}^{(1)} \I \{t>0\} - L_{-t}^{(2)} \I\{t<0\},\]
constructed from the independent compound Poisson processes 
\[L^{(1)}_{t} \eqd L^{(2)}_{t} \eqd \sum_{k=1}^{N_{t}} \xi_{k},\]
where \(N_{t}\) is a Poisson process with intensity \(\lambda\), and \(\xi_{1}, \xi_{2},...\) are independent r.v.'s
with standard exponential distribution. Note that 
 the L{\'e}vy density of the process \(L_{t}^{(1)}\) is  \(\nu(x)=\lambda e^{-x}.\)


 For \(k=1,2\), denote the  jump times of \(L_{t}^{(k)}\) by \(s^{(k)}_{1}, s^{(k)}_{2},...\) and the corresponding jump sizes  by \(\xi^{(k)}_{1}, \xi^{(k)}_{2},...\) Then 
\begin{eqnarray*}
Z_{t} = \sum_{j=0}^{\infty}  \K(t-s^{(1)}_{j})  \xi^{(1)}_{j}
-
\sum_{j=0}^{\infty}  \K(t+s^{(2)}_{j})  \xi^{(2)}_{j}.
\end{eqnarray*}
In practice, we truncate both series in the last representation by finding a value \(x_{max}:= \max_{x\in\R_{+}}\{\K(x)>\alpha\}\) for a given level \(\alpha\). Let
\begin{eqnarray*}
\tilde{Z}_{t} &=& \sum_{k \in K^{(1)}} \K(t-s^{(1)}_{j})  \xi^{(1)}_{j}
-
\sum_{k \in K^{(2)}}  \K(t+s^{(2)}_{j})  \xi^{(2)}_{j}, 
\end{eqnarray*}
where
\begin{eqnarray*}
K^{(1)}&:=& \left\{ k:  \max(0, t - x_{max}) < s^{(1)}_{k} < t+ x_{max}\right\}, \\
K^{(2)}&:=& \left\{ k:  0 < s^{(2)}_{k} < \max(0, -t+x_{max}) \right\}.
\end{eqnarray*}
For simulation study, we take \(\lambda=1,\) \(\alpha=0.01\) (and therefore \(x_{\max}=6.908\)). Typical trajectory of the process \(\tilde{Z}_{t}\) is presented on Figure~\ref{fig1}.
 \begin{figure}
\begin{center}
\includegraphics[width=1\linewidth ]{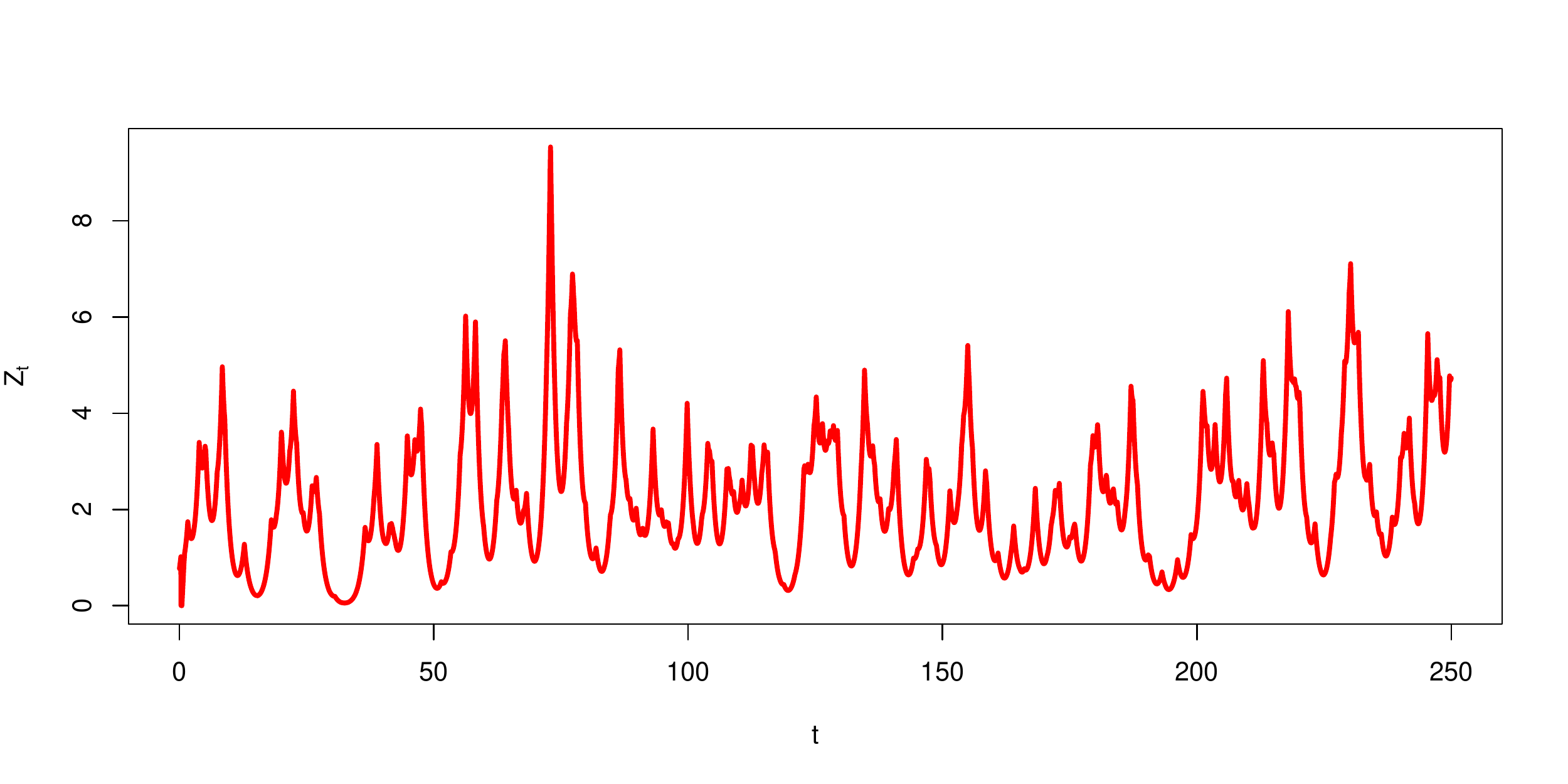}\caption{\label{fig1} Typical trajectory of the process \(Z_{t}\) constructed from the compound Poisson process with positive jumps.}
\end{center}
\end{figure}
\subsection{General idea of the estimation procedure.} In practice  the estimation procedure described in Section~\ref{Mell} can be slightly simplified under the assumption that the L{\`e}vy process \(L\) has no drift. In this case, one can consider the first derivative of the function \(\Psi_{\sigma}(u)\) instead of the second, and get that 

\[
\mathcal{M}[\Psi_{\sigma}'](z)= \widetilde{Q}(z) \cdot\mathcal{M}[\overline{\nu}](1-z), \quad \Re(z)\in (0,1),
\]
where 
\[\widetilde{Q}(z)  = \i \Gamma(z) \exp\{ \i \pi z/2\} 
\int_{\R}
\left(
	K(x)	
\right)^{1-z} dx.\]
The estimation scheme mainly follows the original idea: we first estimate the Mellin transform of the function \(\Psi_{\sigma}'\), and then infer on the L{\'e}vy measure \(\nu\) by applying the Mellin transform techniques. Below we describe these steps in more details. 

\textbf{Estimation of the Mellin transform of \(\Psi'(\cdot)\).} The most natural estimate is
\begin{eqnarray}
\label{MM}
\mathcal{M}_{n}[\Psi'](1-z):=\i  \int_{0}^{U_{n}}
\frac{\mean(Z_{k\Delta} e^{\i u Z_{k \Delta }})}
{
\mean(e^{\i u Z_{k\Delta}})
}
u^{-z} du.
\end{eqnarray}
In order to improve the numerical rates of convergence of the integral involved in \eqref{MM}, we slightly modify this estimate: 
\begin{multline*}
\mathcal{M}_{n}[\Psi'](1-z):=\i \int_{0}^{U_{n}}
\Bigl[
\frac{\mean(Z_{k\Delta} e^{\i u Z_{k \Delta }})}
{
\mean(e^{\i u Z_{k\Delta}})
}
-
\mean(Z) e^{\i u}
\Bigr]
u^{-z}du\\  + 2 \i  \lambda \Gamma(1-z) \exp\{ \i \pi (1-z)/2\}.
\end{multline*}
Note that \(\mathcal{M}_{n}[\Psi'](1-z)\)  is also a consistent estimate of \(\mathcal{M}[\Psi'](1-z)\) (since \(\mean(Z) \to 2 \lambda\)), but involves the integral with better convergence properties.
In our case \(\mathcal{M}[\overline{\nu}](z) = \lambda \Gamma(1+z)\),  and therefore the Mellin transform of the function \(\Psi'\) is equal to 
\begin{eqnarray*}
	\M[\Psi'](1-z) =
\widetilde{Q}(1-z) \cdot\mathcal{M}[\overline{\nu}](z)
=
	2 \i \lambda \frac{ \Gamma(1-z) \Gamma(1+z)}{z} e^{ \i \pi (1-z)/2}.
\end{eqnarray*}
We estimate \(\M[\Psi'](1-z)\) for \(z=c+\i v_k\), where \(c\) is fixed and \(v_k, \; k=1,\ldots, K,\) are taken on the equidistant grid from \((-V_n)\) to \(V_{n}\) with step \(\delta = 2 V_n / K.\) Typical behavior of the  the Mellin transform \(\mathcal{M}[\Psi'](1-z)\) and its estimate \(\mathcal{M}_{n}[\Psi'](1-z)\) is illustrated by Figure~\ref{fig3}. 
 \begin{figure}
\begin{center}
\includegraphics[width=1\linewidth ]{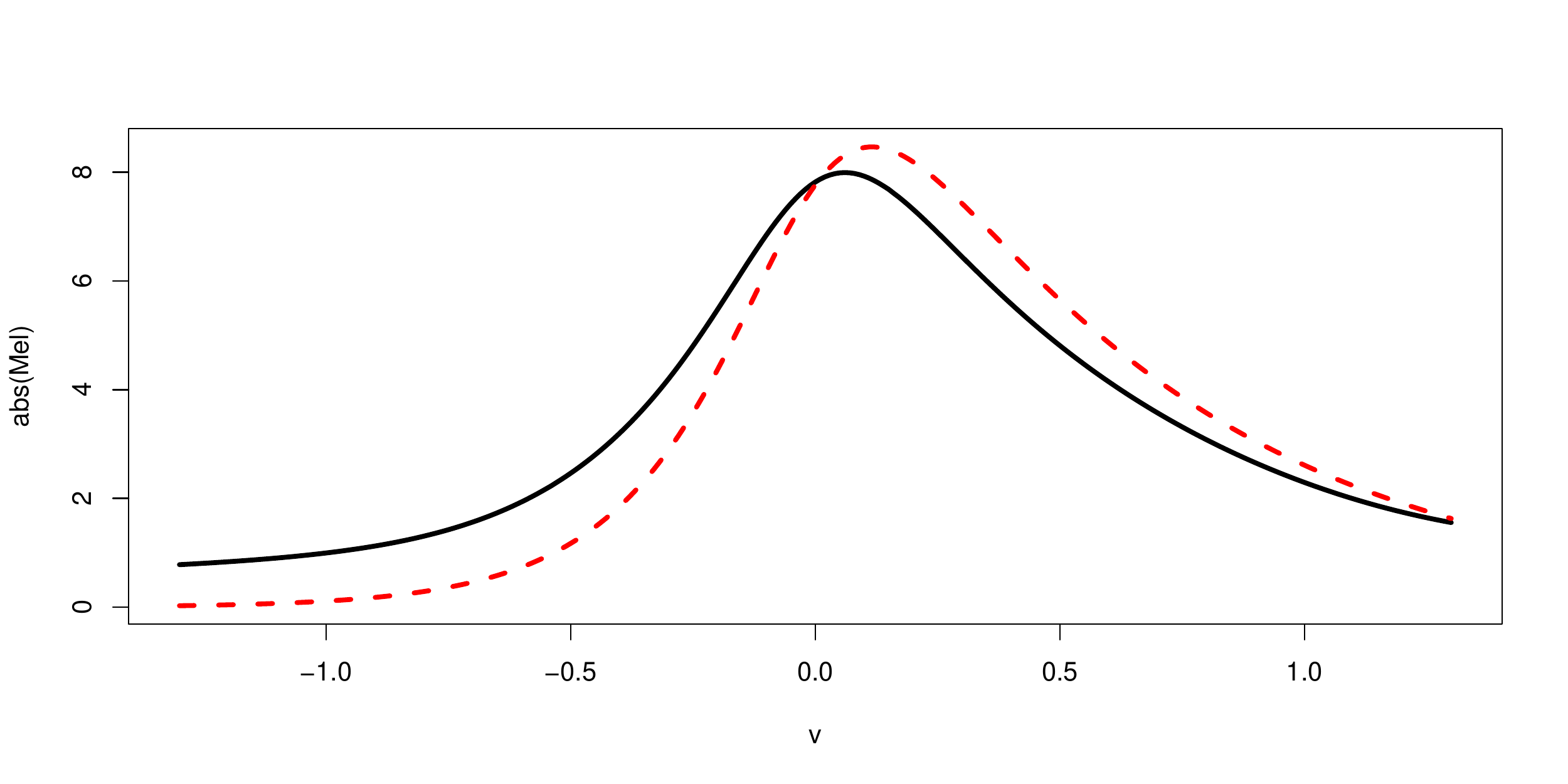}\caption{\label{fig3} 
 Absolute values of the empirical (black solid) and theoretical (red dashed) Mellin transforms of the function $\Psi'(\cdot)$ depending on the imaginary part of the argument.
}
\end{center}
\end{figure}
 \begin{figure}
\begin{center}
\includegraphics[width=1\linewidth ]{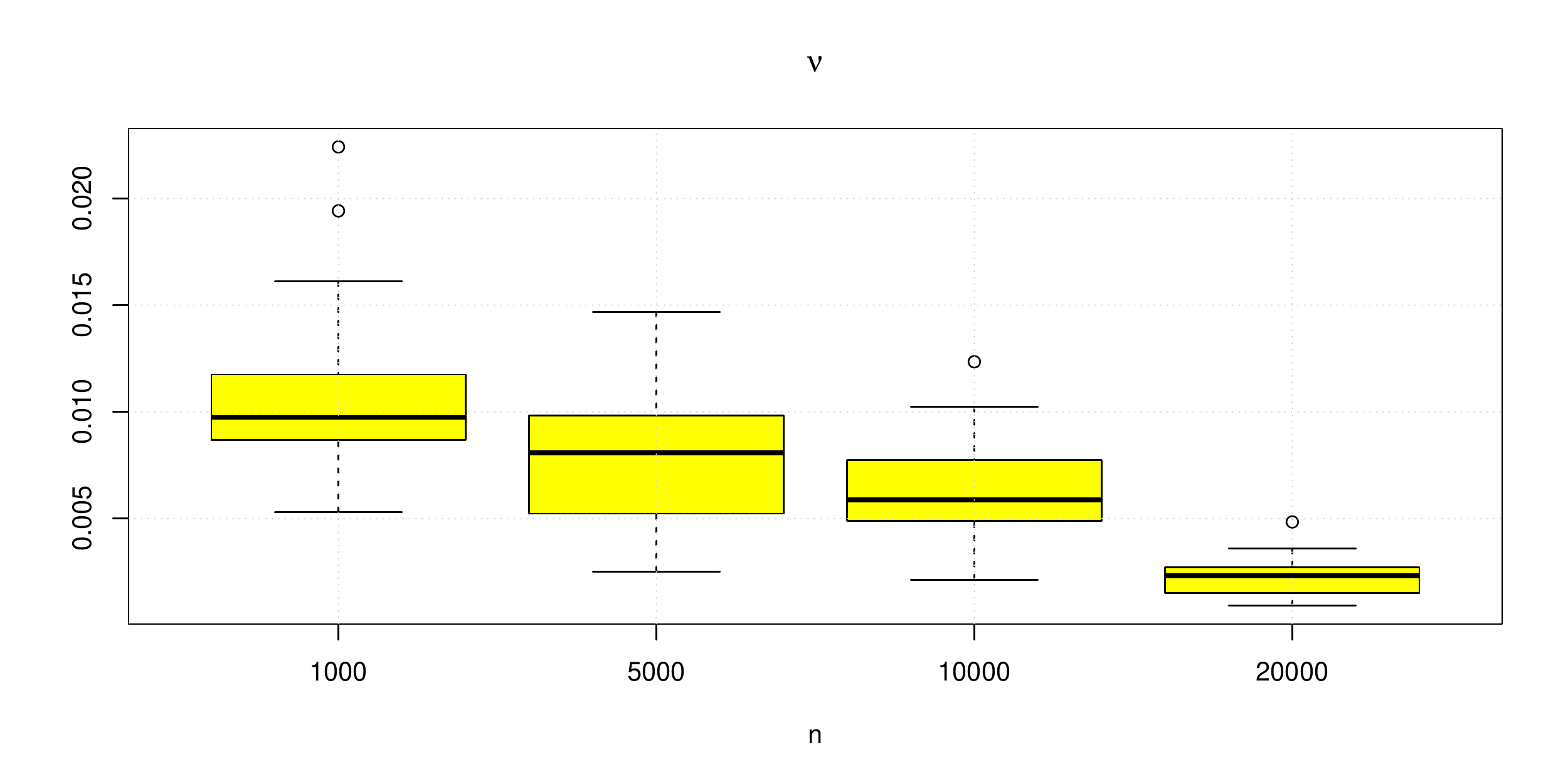}\caption{\label{fig5} 
Boxplot of the estimate \(\RR(\tilde\nu_{n}^{\star})\) based on 20 simulation runs.
}
\end{center}
\end{figure}

%
\textbf{Estimation of \(\nu(x)\).}
Finally, we estimate  the L{\'e}vy density \(\nu(x)\) by 
\[
\tilde{\nu}_{n}(x):=   \frac{\delta}{2 \pi x}
\sum_{k=1}^{K}
\Re \left\{ 
\frac{
 	\mathcal{M}_{n}[\Psi'](1-c-\i v_{k})
} 
{
	\widetilde{Q}(1-c - \i v)}  \cdot
 x^{-(c+\i v_{k})}
 \right\}
\]
and measure  the quality  of this estimate  by the \(L^{2}\)-norm on the interval \([1,3]:\)
\begin{eqnarray*}
	\RR (\tilde\nu_{n}) = 
	\int_{1}^{3} \left(
	\tilde\nu_{n}(x) - \nu(x)
\right)^{2} dx.
\end{eqnarray*}
To show the convergence of this estimate, we made simulations with different values of \(n\). The parameters \(U_n\) and \(V_n\) are chosen by numerical optimization of \(\RR (\tilde\nu_{n})\). The results of this optimization, for different values of \(n\), as well as the means and variances of the estimate \(\tilde\nu_n\) based on 20 simulation runs, are given in the next table.

\begin{center}
\begin{tabular}{ | c | c | c || c | c |}
 \hline
  $n$ & $U_n$ & $V_n$ & $ \mean\left( \RR(\tilde\nu_n) \right)$ & $ \Var\left( \RR(\tilde\nu_n) \right)$\\
  \hline			
    1000 &  0.4  & 1.1 & 0.0109 & $1.62*10^{-5}$ \\
     5000 &  0.4  & 1.2 &  0.0079 & $9.07*10^{-6}$ \\
     10000 & 0.5 & 1.3 & 0.0063 &  $6.56*10^{-6}$\\
     20000 &  0.3 & 1.3 &  0.0023 &  $9.15*10^{-7}$\\
  \hline  
\end{tabular}
\end{center}
The boxplots of this estimate based on 20 simulation runs are presented on Figure~\ref{fig5}. 

\appendix

\section{Proof of Theorem~\ref{thm1}}
Denote $G_{j}(u)=\Psi_{\sigma,n}^{(j)}(u)-\Psi_{\sigma}^{(j)}(u),$ \(j=1,2,\)
where 
\[
\Psi_{\sigma,n}(u)=\log\Phi_{n}(u)+\frac{\sigma^2u^2}{2} \int_{\R} \K^2(x)\,dx.
\]  
Then 
\begin{eqnarray}
G_{1}(u)&=&\frac{D_{1}(u)-D_{0}(u)\Psi_{\sigma}'(u)}{1+D_{0}(u)},
\label{R11}\\
G_{2}(u) & = & \frac{\left(\Psi_{\sigma}''(u)+\left(\Psi_{\sigma}'(u)\right)^{2}+\Psi_{\sigma}'(u)G_{1}(u)\right)D_{0}(u)}{1+D_{0}(u)}
\nonumber\\
 &  & -\frac{\left(2\Psi_{\sigma}'(u)+G_{1}(u)\right)D_{1}(u)}{1+D_{0}(u)}+\frac{D_{2}(u)}{1+D_{0}(u)}.\label{R21}
\end{eqnarray}
We have 
\begin{eqnarray*}
\overline{\nu}_{n}(x) -\overline{\nu}(x) & = & \frac{1}{2\pi \i
}\int_{c-\i V_{n}}^{c+\i V_{n}}\left[\frac{\mathcal{M}_{n}[\Psi_{\sigma}''](1-z)-\mathcal{M}[\Psi_{\sigma}''](1-z)}{Q(1-z)}\right]x^{-z}\, dz\\
 &  & -\frac{1}{2\pi x} \int_{\{|v |\geq V_{n}\}}\mathcal{M}[\overline{\nu}](c+\i v)\, x^{-(c+\i v)}\, dv
\end{eqnarray*}
and  
\begin{eqnarray}\nonumber
x^{c}\left(\overline{\nu}_{n}(x) - \overline{\nu}(x)\right) & = & \frac{1}{2\pi}\int_{\{|v|\leq V_{n}\}}\frac{R_{1}(v)+R_{2}(v)}{Q(1-c-\i v)}x^{-\i v}\, dv\\
\label{nudiff}
 &  & \hspace{0.3cm}-\frac{1}{2\pi}\int_{\{|v|\geq V_{n}\}}\mathcal{M}[\overline{\nu}](c+\i v)\, x^{-\i v}\, dv,
\end{eqnarray}
where
\begin{eqnarray*}
R_{1}(v) := \int_{0}^{U_{n}}G_{2}(u) u^{
-c-\i v }\, du
\end{eqnarray*}
and 
\[
R_{2}(v):=- \int_{U_{n}}^{\infty}\Psi_{\sigma}''(u)\, u^{-c-\i v}\, du.
\]
We have on \(\overline{\mathcal{A}_K},\) under the assumption \(K\varepsilon_n(1+\left\| \Psi_{\sigma}'\right\|_{U_n})\leq1/2,\) that the denominator of the fractions in \(G_{1}\) and \(G_{2}\) can be lower bounded as follows:
\begin{eqnarray*}
\min_{u \in [-U_{n}, U_{n}]} \left|
1 +  D_{0} (u)
\right| 
\geq 
1 - \max_{u \in [-U_{n}, U_{n}]} \left| 
D_{0} (u) 
\right| \geq
1- K \eps_{n}
\geq 1/2.
\end{eqnarray*}
Therefore, 
\begin{eqnarray*}
 \left\| G_{1}\right\|_{U_n}&\leq & 2K\varepsilon_n \left(1+\left\| \Psi_{\sigma}'\right\|_{U_n}\right)\leq 1
\\
\left\| G_{2}\right\|_{U_n}&\leq & 2K\varepsilon_n \left(1+\left\|\Psi_{\sigma}''\right\|_{U_n}+\left\|(\Psi_{\sigma}')^2\right\|_{U_n}\right.
\\
&& +\left.(1+\left\|\Psi_{\sigma}'\right\|_{U_n})\left\| G_{1}\right\|_{U_n}+2\left\|\Psi_{\sigma}'\right\|_{U_n}\right),
\end{eqnarray*}
Thus 
\begin{eqnarray*}
\left|R_{1}(v)\right|\leq 2KU_n^{1-c}\varepsilon_n \left(2+\left\|\Psi_{\sigma}''\right\|_{U_n}+\left\|\Psi_{\sigma}'\right\|^2_{U_n}+3\left\|\Psi_{\sigma}'\right\|_{U_n}\right). 
\end{eqnarray*}
Since 
\[
\Psi_{\sigma}''(u)=  -  \int_{\infty}^{\infty} \K^{2}(x) \cdot \F[\overline{\nu}](u\K(x))\,dx,
\]
it holds  for any \(z \in \C\) 
\begin{eqnarray*}
\int_{U_{n}}^{\infty}\Psi_{\sigma}''(u)u^{-z}\,dy & = &-   \int_{-\infty}^{\infty} \K^{2}(x)\left[\int_{U_{n}}^{\infty}\F[\overline{\nu}](u\K(x))u^{-z}\,du\right]\,dx\\
 & = & -  \int_{-\infty}^{\infty} \left[ \K(x)\right]^{z+1}\left[\int_{U_{n}\K(x)}^{\infty}\F[\overline{\nu}](v)v^{-z}\,dv\right]\,dx.
\end{eqnarray*}
Next,  for any fixed \(x \in \R,\) we can upper bound the inner integral in the right-hand side of the last formula:
\begin{multline*}
\left|\int_{U_{n} \K(x)}^{\infty}\F[\overline{\nu}](v)v^{-z}\,dv\right|\\
\leq\left(1+U_{n}\K(x)\right)^{-\alpha}\cdot
\int_{0}^{\infty}v^{\mathrm{-Re}(z)}(1+v){}^{\alpha}\left|\F[\overline{\nu}](v)\right|\,dv.
\end{multline*}
%
Due to \eqref{assf} 
 we get that for any \(z\) with \(\Re(z)\in (0,1)\) it holds
\begin{eqnarray*}
\int_{0}^{\infty}v^{-\mathrm{Re}(z)}(1+v)^{\alpha}\left|\F[\overline{\nu}](v)\right|\,dv < \frac{\bar\delta}{1-\mathrm{Re}(z)}+A
\end{eqnarray*}
with \(\bar\delta=2^{\alpha} \int_{\R_{+}}x^{2} \nu(x)dx \leq 2^{\alpha}A\) due to \eqref{asbg}.
Finally, we conclude that 
\begin{eqnarray*}
\left|
R_{2}(v)
\right|:= 
\left|
	\int_{U_{n}}^{\infty}\Psi_{\sigma}''(y)y^{-c-\i v}\,dy
\right|
\leq
\left(\frac{\bar\delta}{1-c}+A\right)
\\
\times\int_{\R}
\left[
\K(x)
\right]^{c+1}
\left(1+U_{n}\K(x)\right)^{-\alpha}
dx.
\end{eqnarray*} 
Now an upper bound for the last term in \eqref{nudiff} follows from the  assumption on the Mellin transform of the function \(\bar\nu\). Indeed, since \eqref{assm} is assumed, it holds
\begin{multline*}
	\left| 
		\int_{\{|u|\geq V_{n}\}}\mathcal{M}[\overline{\nu}](c+\i u)\, x^{-\i u}\, du
	\right| 
	\\ \leq
	e^{-\gamma V_{n}}
	\int_{\{|u|\geq V_{n}\}}
	e^{\gamma V_{n}}\left| 
		\mathcal{M}[\overline{\nu}](c+\i u)
	\right| \, du
	 \leq 
	A 	e^{-\gamma V_{n}}.
\end{multline*}
This observation completes the proof.
\section{Proof of Corollary~\ref{cor_xre}}
{
For the sake of simplicity we consider the case \(\rho=1.\) We divide the proof into several steps. For the sake of simplicity we assume that either the kernel \(\K\) is symmetric or is supported on \(\R_+,\) so that it suffices to study the integral over \(\R_+.\)
}
\paragraph{\textbf{1. \textit{Upper bound for \(\Lambda_{n}:=\int_{\R_{+}}
\bigl[
\K(x)
\bigr]^{c+1}
\bigl[1+U_{n}\K(x)\bigr]^{-\alpha}\,
dx\)}}}
Note that the function \(\K(x) = x^{r} e^{-x}\) has two intervals of monotonicity on \(\R_{+}\): \([0, r]\) and \([r, \infty).\) Denote the corresponding inverse functions by \(g_{1}: [0,r^{r} e^{-r}] \to [0,r]\) and \(g_{2}: [0,r^{r} e^{-r}] \to [r,\infty).\) Then 
\begin{eqnarray*}
\Lambda_{n} 
&=& \left(
\int_{0}^{r} + \int_{r}^{\infty}
\right) \left[\K(x)\right]^{c+1}\left[1+U_{n}\K(x)\right]^{-\alpha}\,dx\\
&=&
\int_{0}^{r^{r}e^{-r}}
w^{c+1}\left(
	1+ U_{n} w
\right)^{-\alpha} g_{1}'(w)
dw\\
&& \hspace{2cm}
+
\int_{r^{r}e^{-r}}^{0}
w^{c+1}\left(
	1+ U_{n} w
\right)^{-\alpha} g_{2}'(w)
dw\\
&=&\int_{0}^{r^{r}e^{-r}}
w^{c+1}\left(
	1+ U_{n} w
\right)^{-\alpha} G(w)
dw\\
&=& U_{n}^{-c-2}\Biggl(
\int_{0}^{1}
+
\int_{1}^{r^{r} e^{-r}U_{n}}
\Biggr)
 y^{c+1}\left(1+y\right)^{-\alpha}   \cdot G(y/U_{n})\,dy\\
 && \hspace{7cm}
 =: J_{1}+J_{2},
\end{eqnarray*}
where \(G(\cdot) = g_{1}'(\cdot) - g_{2}'(\cdot).\)  In what follows, we separately analyze the summands \(J_{1}\) and \(J_{2}.\)

\paragraph{\textbf{1a. \textit{Upper bound for \(J_{1}\).}}} Clearly, the behavior of the function \(G(\cdot)\) at zero is crucial for the analysis of \(J_{1}\). Since \(\K(g_{1}(y))=y\) for any \(y \in [0, r^{r} e^{-r}],\) we get \(g_{1}(0)=0\) and moreover as \(y \to 0,\)
\begin{eqnarray*}
g_{1}'(y) = \frac{1}{\K'(g_{1}(y))}=\frac{1}{[g_{1}(y)]^{r-1} e^{-g_{1}(y)} \left(r - g_{1}(y)\right)} \asymp \frac{1} {r [g_{1}(y)]^{r-1}}.
\end{eqnarray*}
Analogously, due to \(\K(g_{2}(y))=y\) for any \(y \in [0, r^{r} e^{-r}],\) we conclude that  \(\lim_{y  \to 0}g_{2}(y)=+\infty\), and as \(y \to 0\)
\begin{multline*}
g_{2}'(y) =\frac{1}{[g_{2}(y)]^{r-1} e^{-g_{2}(y)} \left(r - g_{2}(y)\right)}\\ 
 \asymp \frac{-1} {[g_{2}(y)]^{r}e^{-g_{2}(y)} } = \frac{-1}{\K(g_{2}(y))} = \frac{-1}{y}.
\end{multline*}

For further analysis of the asymptotic behaviour of \(g_{1}(\cdot)\)  we apply the asymptotic iteration method.  We are interested in the   behaviour of the solution \(g_{1}(y)\) of the equation 
\[f(x):=x^{r} e^{-x}-y=0\] 
as \(y \to 0\). Note that  the distinction between the solutions is in the asymptotic behaviour as \(y \to 0\):  \(g_{1}(y) \to 0\), \(g_{2}(y) \to \infty\). Let us iteratively apply the recursion
\begin{eqnarray*}
\varphi_{n+1} = \varphi_{n} - \frac{f(\varphi_{n})}{f'(\varphi_{n})}
=
 \varphi_{n} - \frac{\varphi_{n}^{r}e^{-\varphi_{n}}-y}{\varphi_{n}^{r-1}e^{-\varphi_{n}}\left(
	r -\varphi_{n}
\right)}, \qquad n=1,2,...
\end{eqnarray*}
Motivated by the power series expansion of the function \(e^{-x}\) at zero, 
\begin{eqnarray*}
x^{r}e^{-x} = x^{r} - x^{r+1}+\frac{1}{2} x^{r+2}+ o(x^{r+2}),
\end{eqnarray*}
we take for the initial approximation of \(g_{1}(y)\), the function \(\varphi_{0}= y^{1/r}\).
Then 
\begin{eqnarray*}
\varphi_{1}(y) &=& 
y^{1/r} - \frac{y e^{-y^{1/r}}-y}{y^{(r-1)/r}e^{-y^{1/r}}\left(
	r -y^{1/r}
\right)}
\\
&=&
y^{1/r}
\left( 
1- 
\frac{ e^{-y^{1/r}}-1}{
	e^{-y^{1/r}}\left(
	r -y^{1/r}
\right)}
\right)\\
&=&
y^{1/r}
+O(y^{2/r}).
\end{eqnarray*}
Finally, we conclude that as \(y \to 0\),
\begin{eqnarray*}
G(y) = \frac{1}{r y^{(r-1)/r}}  
\left(
	1 + o(1) 
\right)
+ \frac{1}{y} 
\left(
	1 + o(1) 
\right)
=
\frac{1}{y}
\left(
	1 + o(1) 
\right).
\end{eqnarray*}
Therefore $J_{1}$  can be upper bounded as follows: 
\begin{eqnarray*}
J_{1}
&\leq&
C_{3} U_{n}^{-c-1}\int_{0}^{1}
 y^{c}\left(1+y\right)^{-\alpha}
\left(
1 + o(1)
\right)\,dy.
\end{eqnarray*}
The integral in the right-hand side converges iff \(\int_{0}^{1}y^{c}dy<\infty\). Since \(c\in (0,1),\) we get \(J_{1}\lesssim U_{n}^{-c-1}.\)

\paragraph{\textbf{1b. \textit{Asymptotic behaviour of  \(J_{2}\)}}}
 Analogously, the asymptotic behavior of \(J_{2}\) crucially depends on the behavior of \(G(y)\) at the point \(y = r^{r} e^{-r}.\)  Note that as \(y \to r^{r} e^{-r},\)
\begin{eqnarray*}
g_{k}'(y) = \frac{1}{\K'(g_{k}(y))}=\frac{1}{[g_{k}(y)]^{r-1} e^{-g_{k}(y)} \left(r - g_{k}(y)\right)} \asymp \frac{C}{r-g_{k}(y)}\end{eqnarray*}
for \(k=1,2.\) Taking logarithms of  both parts of the equation \(x^{r} e^{-x} = y\) and changing the variables \(u=x-r\) and \(\delta = r^{r} e^{-r}-y, \) we arrive at the equality 
\begin{eqnarray*}
 u = r \log\left(
	1 + \frac{u}{r}
\right)
-
 \log\left(
	1 - \frac{\delta}{r^{r} e^{-r}}
\right).
\end{eqnarray*}
Consider this equality as \(u \to 0\) and \(\delta \to 0+,\) we get 
\begin{eqnarray*}
u = r\left(
	\frac{u}{r} -  
	\frac{1}{2} \frac{u^{2}}{r^{2}}
\right)
	+
	\frac{\delta}{r^{r} e^{-r}}
	+ O\left(
		\delta^{2}
\right)	+ O(u^{3}),
\end{eqnarray*}
and therefore
\begin{eqnarray*}
u = \pm  \sqrt{
	2 r^{1-r} e^{r}
}
\cdot \sqrt{ \delta}
+
O\left(
		\delta
\right)	+ O(u^{3/2})
\end{eqnarray*}
corresponding to the functions \(g_{1}\) and \(g_{2}\). Finally, we conclude 
\begin{eqnarray*}
|G(y)| \asymp \frac{C \sqrt{2} }{\sqrt{
	 r^{1-r} e^{r}
}
}
\frac{1}{
\sqrt{
	r^{r} e^{-r} - y
}
}, \qquad y \to r^{r} e^{-r},
\end{eqnarray*}
and therefore
\begin{eqnarray*}
J_{2} &\sim &
 U_{n}^{-c-3/2}\int_{1}^{r^{r} e^{-r}U_{n}}
 y^{c+1}\left(1+y\right)^{-\alpha}   \cdot\frac{1}{
\sqrt{
	r^{r} e^{-r}U_{n} - y
}
}\,dy.
\end{eqnarray*}
We change the variable in the last integral: \[z=\sqrt{
\frac{
	r^{r} e^{-r}U_{n} - 1
}{
	r^{r} e^{-r}U_{n} - y
}
}, \qquad
y= r^{r} e^{-r}U_{n}
+ 
 \frac{1-r^{r} e^{-r}U_{n}}{z^{2}},
 \] and get with \(\widetilde{U}_{n}=r^{r} e^{-r}U_{n}\)
\begin{multline*}
J_{2} \asymp 
 U_{n}^{-c-3/2}
 \int_{1}^{\infty}
 \left(
\wU
+ 
 \frac{1-\wU}{z^{2}}
\right)^{c+1}\\
\cdot \left(
1 + \wU
+ 
 \frac{1-\wU}{z^{2}}
\right)^{-\alpha} 
\cdot\frac{z}{\sqrt{\wU -1}}
\frac{2(\wU -1)}{z^{3}}dz.
\end{multline*}
Therefore,
\begin{eqnarray*}
J_{2} \asymp C_{4} U_{n}^{-c-3/2} 
\wU^{c+1} \left(
	\wU+1
\right)^{-\alpha}
\sqrt{\wU-1}, \qquad n \to \infty,
\end{eqnarray*}
with some constant \(C_{4}>0\)
and we conclude that 
\(J_{2} \asymp  C_{5} U_{n}^{-\alpha}\) as \(n \to \infty.\) To sum up, \(\Lambda_{n} \lesssim U_{n}^{-\min(\alpha,c+1)} = U_{n}^{- \alpha}\) as \(n \to \infty.\)

\paragraph{\textbf{2. \textit{Upper bound for \(H_{n}:=\int_{\{|v|\leq V_{n}\}}\left|Q(1-c-\i v)\right|^{-1}\, dv\)}}} Recall that 
\begin{eqnarray*}
H_{n}=\int_{\{|v|\leq V_{n}\}}\frac{e^{-\pi v /2}}
{
\bigl|
	\Gamma(1-c-\i v) 
\bigr|
\cdot
\bigl|
	\int_{\R}
	\left(
		\mathcal{K}(x)	
	\right)^{c+1 + \i v}	 
	dx
\bigr|
}\, dv
\end{eqnarray*}

Note that for our choice of the function \(\K(\cdot)\), it holds for any \(z \in \C\)
\begin{eqnarray*}
\int_{\R}
\left(
	K(x)	
\right)^{z}	 dx
=
 2 \int_{\R_{+}} 
(x^{r} e^{- x})^{z}
dx =
2 
\left[ 
	\lim _{R \to +\infty}\int_{\gamma_{R}(z)} u^{rz} e^{- u} du 
\right]
\cdot z^{-(rz+1)},
\end{eqnarray*}
where \(\gamma_{R}(z)\) is the part of the complex line  \(\left\{(x \Re(z), x \Im(z)), \; x  \in [0,R]\right\}\). Note that due to the Cauchy theorem,  for any \(z\) with positive real part
\begin{eqnarray}
\label{RRR}
\int_{\R_{+}} u^{rz} e^{-\rho u} du =
	\lim _{R \to +\infty}\int_{\gamma_{R}(z)} u^{rz} e^{- u} du 
	+ 
	\lim _{R \to +\infty}\int_{c_{R}} u^{rz} e^{- u} du 
\end{eqnarray}
with \(c_{R} := \left\{
\left(R \cos(\theta), R \sin(\theta) \right), \theta \in (0, \arctan(\Im(z)/\Re(z))
\right\}\). Since the last limit in 
\eqref{RRR} is equal to 0, we conclude that 
\[\int_{\R}
\left(
	K(x)	
\right)^{c+1 + \i v} dx = 2  \; \Gamma\Bigl(
	r(c+1)+1+\i v r 
	\Bigr) \cdot e^{-(r(c+1)+1+ \i v r) \cdot \log(c+1 + \i v)}.\]
Next, using the fact that there exists a constant \(\bar{C}>0\) such that \(|\Gamma( \alpha + \i \beta) | \geq \bar{C}|\beta|^{\alpha-1/2} e^{-|\beta| \pi/2}\) for any \(\alpha\geq-2, |\beta|\geq 2\) (see Corollary~7.3 from \cite{BSM}), we get that 
\begin{eqnarray*}
\frac{e^{-\pi v /2}}{\left| \Gamma(1-c-\i v)
 \right|} \leq v^{c-1/2},
\end{eqnarray*}
and moreover
\begin{eqnarray*}
\left|
\int_{\R}
\left(
	K(x)	
\right)^{c +1+ \i v} dx 
\right| &=&
2 
\frac{
\left| \Gamma(r(c+1)+1+\i v r ) 
\right|
} 
{
\left(
	(c+1)^{2} + v^{2}
\right)^{(r(c+1)+1)/2}
e^{-vr \arctan(v/(c+1))}
}.
\end{eqnarray*}
The asymptotic behavior of the  last expression depends on the value \(r.\) More precisely, 
\begin{eqnarray*}
\left|
\int_{\R}
\left(
	K(x)	
\right)^{c + 1 +  \i v} dx 
\right| 
\sim
\begin{cases}
2 
\frac{ c (vr)^{r(c+1) +1/2} e^{-vr \pi/2}
} 
{
\left(
	(c+1)^{2} + v^{2}
\right)^{(r(c+1)+1)/2}
e^{-vr \arctan(v/(c+1))}
}
\sim v^{-1/2},\\
\hspace{3cm} \text{if $r=1,2,...$,}\\
v^{-1},
\hspace{2.35cm} \text{if $r=0$}.
	\end{cases}
\end{eqnarray*}
as \(v \to +\infty.\) Finally, we conclude that 
\(
H_{n} \lesssim V_{n}^{c + 1 },
\) if $r=1,2,...$, and \(H_{n} \lesssim V_{n}^{c+3/2}  \) if \(r=0.\)

\section{Mixing properties of the L\'evy-based MA processes}
\begin{theorem}
\label{mixing_prop}
Let $\left(L_{t}\right)$ be a L{\'e}vy process
with L{\'e}vy triplet  $(\mu,\sigma^2,\nu),$ where \(\sigma>0\) and $\mathrm{supp}(\nu)\subseteq\mathbb{R}_{+}.$ Consider a L{\'e}vy-based moving average
process of the form 
\[
Z_{s}=\int \K(s-t)\,dL_{t},\quad s\geq 0
\]
with a nonegative kernel $\K$. Fix some $\Delta>0$ and denote 
\[
Z_{S}:=\left(Z_{j\Delta}\right)_{j\in S}
\]
for any subset $S$ of $\{1,\ldots,n\}.$ Fix two natural numbers
$m$ and $p$ such that $m+p\leq n.$ For any subsets $S\subseteq\{1,\ldots,m\}$
and $S'\subseteq\{p+m,\ldots,n\},$ let $g$ and $g'$ be two real
valued functions on $\mathbb{R}^{\left|S\right|}$ and $\mathbb{R}^{\left|S'\right|}$
satisfying 
\[
\max\left\{ \left\Vert e^{-R_{S}^{\top}\cdot}g\right\Vert _{L^{1}},\left\Vert e^{-R{}_{S'}^{\top}\cdot}g'\right\Vert _{L^{1}}\right\} <\infty
\]
for some $R_{S}\in\mathbb{R}_{+}^{\left|S\right|}$and $R_{S'}\in\mathbb{R}_{+}^{\left|S'\right|},$ and denote \(C_{\circ}:=\left\Vert e^{-R_{S}^{\top}\cdot}g\right\Vert _{L^{1}} \cdot\left\Vert e^{-R{}_{S'}^{\top}\cdot}g'\right\Vert _{L^{1}}.\)
Suppose that the Fourier transform $\widehat{\K}$ of $\K$ fulfils 
\[
K^{*}:=\sum_{j=-\infty}^{\infty}\left|\widehat{\K}\left(2\pi\frac{j}{\Delta}\right)\right|<\infty
\]
and 
\[
\int_{\left|x\right|>1}e^{R^{*}x}x^2\nu(dx)\leq A_{R^*}
\]
for $R^{*}=\frac{\left\Vert R_{S\cup S'}\right\Vert _{\infty}K^{*}}{\Delta}.$
Then 
\begin{eqnarray}
\label{eq: covz}
\left|\mathrm{Cov}\left(g(Z_{S}),g'(Z_{S'})\right)\right| & \leq & C_{R}C_{\circ}\max_{|l|>p}\left(\K\star \K\right)(l\Delta)\\
\nonumber
 &  & \times\int\|u_{S\cup S'}-\i R_{S\cup S'}\|^{2}\exp\left(-\sigma^2\lambda_{S\cup S'}(u)\right)du_{S\cup S'},
\end{eqnarray}
where $\lambda_{S}(u):=\sum_{k,j\in S}u_k u_{j}(\mathcal{K}\star\mathcal{K})(\Delta (k-j))$ for any \(u\in \R^n\) and \(C_R=\exp(\sigma^2\lambda_{S\cup S'}(R_{S\cup S'})).\) 
\end{theorem}
\begin{proof}
We have for any $S\subseteq\{1,\ldots,n\}$ 
\begin{eqnarray*}
\Phi_{S}(u_{S}-\i R_{S}) & := & \E\left[\exp\left(i\sum_{j\in S}u_{j}Z_{j\Delta}+\sum_{j\in S}R_{j}Z_{j\Delta}\right)\right]\\
 & = & \exp\left(\int\psi\left(\sum_{j\in S}\left(u_{j}-\i R_{j}\right)\K(t-j\Delta)\right)\,dt\right),
\end{eqnarray*}
where $u_{S}:=(u_{j}\in\mathbb{R},\,j\in S)$ and $R_{S}:=(R_{j}\in\mathbb{R}_{+},\,j\in S),$
provided 
\[
\E\left[\exp\left(\sum_{j\in S}R_{j}Z_{j\Delta}\right)\right]<\infty.
\]
 Denote for any subsets $S\subseteq\{1,\ldots,m\}$ and $S'\subseteq\{p+m,\ldots,n\},$
\begin{multline*}
D(u_{S}-\i R_{S},u_{S'}-\i R_{S'}) \\ :=  \Phi_{S,S'}(u_{S}-\i R_{S},u_{S'}-\i R_{S'})-\Phi_{S}(u_{S}-\i R_{S})\Phi_{S'}(u_{S'}-\i R_{S'}),
\end{multline*}
where it is assumed that 
\[
\E\left[\exp\left(\sum_{j\in S\cup S'}R_{j}Z_{j\Delta}\right)\right]<\infty
\]
Then using the elementary  inequality $\left|e^{z}-e^{y}\right|\leq\left(\left|e^{z}\right|\vee\left|e^{y}\right|\right)\left|y-z\right|,$
$y,z\in\mathbb{C},$ we derive
\begin{multline*}
\left|D(u_{S}-\i R_{S},u_{S'}-\i R_{S'})\right| \\ \leq  \left\{ \left|\Phi_{S,S'}(u_{S}-\i R_{S},u_{S'}-\i R_{S'})\right|\vee\left|\Phi_{S}(u_{S}-\i R_{S})\Phi_{S'}(u_{S'}-\i R_{S'})\right|\right\} \times\\
 \left|\int\left\{ \psi\left(\sum_{j\in S\cup S'}\left(u_{j}-\i R_{j}\right)\K(x-j\Delta)\right)-\psi\left(\sum_{j\in S}\left(u_{j}-\i R_{j}\right)\K(x-j\Delta)\right)\right.\right.\\
 \left.\left.-\psi\left(\sum_{j\in S'}\left(u_{j}-\i R_{j}\right)\K(x-j\Delta)\right)\right\} \,dx\right|.
\end{multline*}
Due to Lemma \ref{lem:psi} and the Poisson summation formula, we
derive 
\begin{multline*}
\left|D(u_{S}-\i R_{S},u_{S'}-\i R_{S'})\right| \\ \leq  \left\{ \left|\Phi_{S,S'}(u_{S}-\i R_{S},u_{S'}-\i R_{S'})\right|\vee\left|\Phi_{S}(u_{S}-\i R_{S})\Phi_{S'}(u_{S'}-\i R_{S'})\right|\right\} \times\\
 \left[\sum_{j\in S}\sum_{l\in S'}\left|\left(u_{l}-\i R_{l}\right)\left(u_{j}-\i R_{j}\right)\right|\left(\K\star \K\right)((j-l)\Delta)\right]\\
 \times\int y^{2}e^{\frac{y\left\Vert R\right\Vert _{\infty}K^{*}}{\Delta}}\,\nu(dy).
\end{multline*}
We have 
\begin{multline*}
\mathrm{Cov}\left(g(Z_{S}),g'(Z_{S'})\right)\\=\int_{\mathbb{R}_{+}^{\left|S\right|}}\int_{\mathbb{R}_{+}^{\left|S'\right|}}g(x_{S})g'(x_{S'})\left(p_{S,S'}(x_{S},x_{S'})-p_{S}(x_{S})p_{S'}(x_{S'})\right)\,dx_{S}\,dx_{S'}.
\end{multline*}
and the Parseval's identity implies
\begin{eqnarray*}
\mathrm{Cov}\left(g(Z_{S}),g(Z_{S'})\right) & = & \frac{1}{(2\pi)^{\left|S\right|+\left|S'\right|}}\int_{\mathbb{R}^{\left|S\right|}}\int_{\mathbb{R}^{\left|S'\right|}}\widehat{g}(\i R_{S}-u_{S})\widehat{g}(\i R_{S'}-u_{S'}).\\
 &  & \times D(u_{S}-\i R_{S},u_{S'}-\i R_{S'})\,du_{S}\,du_{S'},
\end{eqnarray*}
\(\widehat{g}\) stands for the Fourier transform of \(g.\)
Hence 
\begin{multline*}
\left|\mathrm{Cov}\left(g(Z_{S}),g'(Z_{S'})\right)\right| \\ \leq  \frac{C_{\c\i Rc}}{(2\pi)^{\left|S\right|+\left|S'\right|}}\int_{\mathbb{R}^{\left|S\right|}}\int_{\mathbb{R}^{\left|S'\right|}}\left|D(u_{S}-\i R_{S},u_{S'}-\i R_{S'})\right|\,du_{S}\,du_{S'}.
\end{multline*}
Furthermore, for any set $S\in\{1,\ldots,n\},$ we have 
\begin{eqnarray*}
\int\psi\left(\sum_{j\in S}\left(u_{j}-\i R_{j}\right)\K(s-j\Delta)\right)\,ds & \leq & -\sigma^2\lambda_{S}(u) + \sigma^2\lambda_{S}(R).
\end{eqnarray*}
 As a result 
\[
\left|\Phi_{S}(u_{S}-\i R_{S})\right|\leq C_{R}\exp\left(-\sigma^2\lambda_{S}(u) \right)
\]
 and 
\begin{eqnarray*}
\left|D(u_{S}-\i R_{S},u_{S'}-\i R_{S'})\right| & \leq & \max_{|l|>p}\left(\K\star \K\right)(l\Delta)\sum_{j\in S}\sum_{l\in S'}\left|\left(u_{l}-\i R_{l}\right)\left(u_{j}-\i R_{j}\right)\right|\\
 &  & C_R\,\exp\left(-\sigma^2\lambda_{S\cup S'}(u)\right).
\end{eqnarray*}
\end{proof}
\begin{lem}
\label{lem:psi}Set
\[
\psi(z)=\int_{0}^{\infty}(\exp(zx)-1)\nu(dx)
\]
 for any $z\in\mathbb{C},$ such that the integral $\int_{|x|>1}\exp(\mathrm{Re}(z)x)\nu(dx)$
is finite. Then

\[
\left|\psi(z_{1}+z_{2})-\psi(z_{1})-\psi(z_{2})\right|\leq2\left|z_{1}\right|\left|z_{2}\right|\int x^{2}e^{x\left(\mathrm{Re}(z_{1})+\mathrm{Re}(z_{2})\right)}\nu(dx),
\]
provided the integral $\int x^{2}e^{x\left(\mathrm{Re}(z_{1})+\mathrm{Re}(z_{2})\right)}\nu(dx)$
is finite.\end{lem}
\begin{proof}
We have 
\begin{multline*}
\psi(z_{1}+z_{2})-\psi(z_{1})-\psi(z_{2}) \\
= \int_{0}^{\infty}(\exp((z_{1}+z_{2})x)-\exp(z_{1}x)-\exp(z_{2}x)+1)\nu(dx)
\\
=
 \int_{0}^{\infty}(\exp(z_{1}x)-1)(\exp(z_{2}x)-1)\nu(dx).
\end{multline*}
Since 
\begin{eqnarray*}
\left|\exp(z)-1\right| & = & \left|e^{\mathrm{Re}(z)}e^{i\mathrm{Im}(z)}-1\right|\\
 & = & \left|e^{\mathrm{Re}(z)}\left(e^{i\mathrm{Im}(z)}-1\right)+e^{\mathrm{Re}(z)}-1\right|\\
 & \leq & \left|\mathrm{Im}(z)\right|e^{\mathrm{Re}(z)}+\left|e^{\mathrm{Re}(z)}-1\right|\\
 & \leq & \left(\left|\mathrm{Re}(z)\right|+\left|\mathrm{Im}(z)\right|\right)e^{\mathrm{Re}(z)}\\
 & \leq & \sqrt{2}\left|z\right|e^{\mathrm{Re}(z)},
\end{eqnarray*}
we get 
\begin{eqnarray*}
\left|\psi(z_{1}+z_{2})-\psi(z_{1})-\psi(z_{2})\right| & \leq & \int_{0}^{\infty}\left|\exp(z_{1}x)-1\right|\left|\exp(z_{2}x)-1\right|\nu(dx)\\
 & \leq & 2\left|z_{1}\right|\left|z_{2}\right|\int x^{2}e^{x\left(\mathrm{Re}(z_{1})+\mathrm{Re}(z_{2})\right)}\nu(dx).
\end{eqnarray*}
\end{proof}
\begin{lem}\label{xre}
Let \(\K(x) = |x|^{r} e^{-\rho |x|}\) with some \(r \in \N \cup \{0\}\) and \(\rho>0.\)  Then 
\begin{eqnarray}
\label{intR}
	 \frac{(\mathcal{K}\star\mathcal{K})(\Delta (k-j))}{(\mathcal{K}\star\mathcal{K})(0)}
	 \leq 
\kappa_{0} \; (j-k)^{\kappa_{1}} e^{-\kappa_{2} (j-k)}
\end{eqnarray}
for all \(j>k\) with \(\kappa_{2}= \Delta \rho, \; \kappa_{1}=2r+1,\) and
\begin{eqnarray*}
\kappa_{0} &=& \frac{(2r+3)}{2} \max\left\{
\frac{\Delta^{2r+1}}{2^{2 r}}, 
\max_{m=0,\ldots, r}\left\{
	C_{r}^{m}
\frac{(r+m)!}{(2r)!} (2\rho\Delta)^{r-m}
\right\}\right\} 
\end{eqnarray*}
with \(C_{r}^{m}={r \choose m}.\)
Moreover, all eigenvalues of the matrix \( ((\mathcal{K}\star\mathcal{K})(\Delta (k-j)))_{k,j\in \mathbb{Z}}\) are bounded from below and above by two finite positive numbers, provided \(\kappa_{2}\) (equivalently \(\rho\)) is large enough. 
\end{lem}
\begin{proof}
We have
\begin{eqnarray*}
(\mathcal{K}\star\mathcal{K})(0)=2\int_{0}^{\infty}x^{2r}e^{-2\rho x}\,dx=2(2\rho)^{-2r-1}\Gamma(2r+1)
\end{eqnarray*}
and 
\begin{multline*}
	 \int_{\R} 
	 	\K_{\Delta j}(v) \K_{\Delta k}(v)\,  
dv  = \left( 
	\int_{-\infty}^{\Delta k}  
	+
	\int_{\Delta k} ^{\Delta j}
	+
	\int_{\Delta j}^{\infty}
\right)
	 	\K_{\Delta j}(v) \K_{\Delta k}(v)\,  
dv \\=: I_{1} + I_{2} +I_{3},
\end{multline*}
where \(\K_{t}(s) := \K(s-t), \; \forall s,t \in \R_{+}\) .
In the sequel we separately consider  integrals \(I_{1}, I_{2}, I_{3}\).
We have
\begin{eqnarray*}
I_{1} &=& \int_{\Delta j} ^{\infty}
\left(
	v - \Delta j
\right)^{r}
\left(
	v - \Delta k
\right)^{r}
e^{-2\rho v+\Delta\rho (j+k)}
dv \\
 &=& \int_{\R_{+}}
	u^{r}
\left(
	u + \Delta (j-k)
\right)^{r}
e^{-2\rho u-\rho \Delta(j-k)}
du\\
 &=& e^{-\rho\Delta(j-k)} \int_{\R_{+}}
	u^{r}
\left(
	\sum_{m=0}^{r} C_{r}^{m} u^{m} 
	\left(
		\Delta (j-k)
	\right)^{r-m}
\right)
e^{-2\rho u}
du\\
&=&
\left[ 
\sum_{m=0}^{r} C_{r}^{m}
(r+m)! 
\frac{
	\Delta^{r-m}
}{
	(2\rho)^{r+m+1}
}
(j-k)^{r-m}\right] e^{-\rho\Delta(j-k)},
\end{eqnarray*}
because \(\int_{\R_{+}} u^{r+m} e^{-2\rho u}du=2^{-(r+m+1)} \Gamma(r+m+1) = (2\rho)^{-(r+m+1)}  (r+m)!.\)
\begin{multline*}
I_{2} = \int_{\Delta k} ^{\Delta j}
\left[ 
- 
\left(
	v - \Delta j
\right)
\left(
	v - \Delta k
\right)
\right]^{r}
e^{-\rho\Delta(j-k)}
dv \\ \leq \frac{\Delta^{2r+1}}{2^{2 r}} \left(
	j-k
\right)^{2r +1 }
e^{-\rho \Delta \left(
	j-k
\right)},
\end{multline*}
because maximum of the quadratic function \(f(v) := - \left(
	v - \Delta j
\right)
\left(
	v - \Delta k
\right)\) is attained at the point \(v=\Delta \left(
	k+j
\right)/2\) and is equal to \(\left(
	\Delta^{2}/4
\right) \left(j-k\right)^{2}.\)
\begin{eqnarray*}
I_{3} &=& \int_{-\infty} ^{\Delta k}
\left( \Delta j -v
\right)^{r}
\left(
 	\Delta k - v
\right)^{r}
e^{2\rho v-\rho\Delta(j+k)}
dv =\\
 &=& \int_{\R_{+}}
 \left(
	u + \Delta (j-k)
\right)^{r}
	u^{r}
e^{-2\rho u-\rho\Delta(j-k)}
du = I_{1}.
\end{eqnarray*}
Next, the well-known Gershgorin circle theorem implies that the minimal eigenvalue of the matrix \( ((\mathcal{K}\star\mathcal{K})(\Delta (k-j)))_{k,j\in \mathbb{Z}}\) is bounded from below by
\[
 (\mathcal{K}\star\mathcal{K})(0)-2\sum_{l>0}(\mathcal{K}\star\mathcal{K})(l)=
  (\mathcal{K}\star\mathcal{K})(0)\left[1-2\kappa_{0} \sum_{l>0}\; l^{\kappa_{1}} e^{-\kappa_{2} l}\right].
\]
Note that for any natural number \(\kappa_{1}>0\)
\[
\sum_{l\geq 1}l^{\kappa_{1}} e^{-\kappa_2 l}=(-1)^{\kappa_{1}}\left.\frac{d^{\kappa_{1}}}{dx^{\kappa_{1}}}\left(\frac{e^{-x}}{1-e^{-x}}\right)\right |_{x=\kappa_2}.
\]
Hence the minimal eigenvalue of the matrix \( ((\mathcal{K}\star\mathcal{K})(\Delta (k-j)))_{k,j\in \mathbb{Z}}\) is bounded from below by a positive number, if \(\kappa_2\) is large enough. Analogously the maximal eigenvalue of the matrix \( ((\mathcal{K}\star\mathcal{K})(\Delta (k-j)))_{k,j\in \mathbb{Z}}\) is bounded from above by 
\begin{eqnarray*}
(\mathcal{K}\star\mathcal{K})(0)+2\sum_{l>0}(\mathcal{K}\star\mathcal{K})(l)=
  (\mathcal{K}\star\mathcal{K})(0)\left[1+2\kappa_{0} \sum_{l>0}\; l^{\kappa_{1}} e^{-\kappa_{2} l}\right]
\end{eqnarray*}
which  is finite.
\end{proof}
\section{Proof of Theorem~\ref{ak_bound}}
The rest of the proof of Theorem~\ref{ak_bound} basically follows the same lines as the proof of Proposition~3.3 from \cite{BelReiss}. First note that 
\begin{multline*}
\max_{|u| \leq U_{n}} \frac{
	\left| 
		\Phi_{n}(u) - \Phi(u)
	\right| 
}{
	\left|
		\Phi(u)
	\right|
}
\leq 
\exp\left\{C_{1} \sigma^2 U_n^2 \int_{\mathbb{R}} (\mathcal{K}(x))^2 dx \right\} 
\cdot 
\max_{|u| \leq U_{n}} 
	\Bigl| 
		\Phi_{n}(u) - \Phi(u)
	\Bigr|
\end{multline*}
for \(n\) large enough.
Next, we separately consider the real and imaginary parts of the difference between \(\Phi_{n}(u)\) and \(\Phi(u).\) Denote
\begin{eqnarray*}
S_{n}(u) :=n  \Re \left(
\Phi_{n}(u) - \Phi(u)
\right)= 
\sum_{k=1}^{n} \left[
	\cos\left(
		u Z_{k \Delta}
	\right)
	-
	\E\left[
		\cos\left(
		u Z_{k\Delta}
	\right)
	\right]
\right]
\end{eqnarray*}
Since \(S_{n}(u)\) is a sum of  centred real-valued random variables, bounded by \(2\) and satisfying \eqref{eq: covz} with \eqref{intR}, there exist a positive constant \(c_1\) such that 
\begin{eqnarray}\label{MPR}
\P \left\{
	\left| 
		S_{n}(u) 
	\right| 
	\geq x 
\right\} \leq 
\exp\left\{
	\frac{
		-c_1 x^{2}
	}
	{
		2n + x \log(n) \log\log(n)
	}
\right\}, \quad \forall \; x\geq 0,
\end{eqnarray}
see Theorem~1 from \cite{MPR}. In order to apply now the classical chaining argument, we divide the interval \([-U_{n},U_{n}]\) by \(2J\) equidistant points \(\left(u_{j}\right)=:\mathcal{G}\), where \(u_{j} = U_{n}(-J+j)/J\), \; \(j=1,\ldots,2J.\) Applying \eqref{MPR}, we get for any \(x\geq 0,\)
\begin{eqnarray}
\label{MPR2}
\P \left\{
	\max_{u_{j} \in \G}
	|S_{n}(u_{j})| \geq x/2
\right\}
\leq 
2J
\exp\left\{
	\frac{
		-c_1 x^{2}
	}
	{
		8n + 2 x \log(n) \log\log(n)
	}
\right\}.
\end{eqnarray}
Note that for any \(u \in [-U_{n}, U_{n}]\) there exists a point \(u^{\star}\in \G\) such that
\(\left|
	u-u^{\star}
\right|  \leq U_{n}/J\)
and therefore
  for all \(k \in 1,\ldots, n,\)
\begin{eqnarray*}
\left|
	\cos(u Z_{k \Delta}) - \cos(u^{\star} Z_{k\Delta})
\right|
\leq 
\left| 
	Z_{k \Delta}
\right| 
\cdot 
\left|
	u-u^{\star}
\right| \leq 
\left| 
	Z_{k \Delta}
\right| 
\cdot 
U_{n}/J.
\end{eqnarray*}
Next, we get
\begin{multline*}
\P \biggl\{
	\max_{| u | \leq U_{n}}
	\left| 
		S_{n}(u) 
	\right| 
	\geq x 
\biggr\} \\ \leq 
\P \biggl\{
	\max_{u_{j} \in \G}
	|S_{n}(u_{j})| \geq x/2
\biggr\}
+ 
\P \biggl\{\sum_{k=1}^{n}
 \left(
|Z_{k\Delta}| + \E \left[
| Z_{k\Delta} |
\right]
\right)
U_{n} /J\geq x/2
\biggr\}.
\end{multline*}
Applying \eqref{MPR2} and the Markov inequality, we arrive at
\begin{multline*}
\P \biggl\{
\max_{| u | \leq U_{n}}
	\left| 
		S_{n}(u) 
	\right| 
	\geq x 
\biggr\} 
\\ \leq 
2J
\exp\left\{
	\frac{
		-c_1 x^{2}
	}
	{
		8n + 2 x \log(n) \log\log(n)
	}
\right\}
+ 
\frac{
	4 U_{n}
}{
	x J
} 
n
\E
 	\left| Z_{\Delta}\right|
,
\end{multline*}
where \(\E\left[
 	|Z_{\Delta}|
\right] \leq  \left(\E\left[
 	|Z_{\Delta}|^2
\right]\right)^{1/2}\) is finite due to \eqref{sec_mom}. The choice 
\begin{eqnarray*}
J &=&  \floor \left(
\sqrt{\frac{U_{n} n}{x} \cdot \exp \left\{
	\frac{c_1 x^{2}}{
		8 n + 2x \log(n) \log\log(n)
		}
\right\}}
\right),
\end{eqnarray*}
where \(\floor(\cdot)\) stands for the largest integer smaller than the argument,
leads to the estimate
\begin{eqnarray*}
\P \biggl\{
\max_{| u | \leq U_{n}}
	\left| 
		S_{n}(u) 
	\right| 
	\geq x 
\biggr\} 
&\leq& c_{2} \sqrt{\frac{U_{n} n}{x}}
\exp\left\{
	\frac{
		-c_1 x^{2}
	}
	{
			16n + 4 x \log(n) \log\log(n)
	}
\right\}\\
&\leq&
c_{2} \sqrt{\frac{U_{n} n}{x}}
\exp\left\{
	\frac{-c_{3}x^{2}}{n}
\right\},
\end{eqnarray*}
which holds for \(n\) large enough with \(c_{2}=2 \left(
	1 + \E\left[|
		Z_{\Delta}|
	\right]
\right), c_{3}= c_1/17,\) provided \(x \lesssim n^{1-\eps}\) with some \(\eps>0.\) Finally, 
\begin{multline*}
\P \biggl\{
\max_{| u | \leq U_{n}}
	\left| 
		S_{n}(u) 
	\right| 
	\geq x 
\biggr\} 
\\ \geq
\P \left\{
\max_{|u| \leq U_{n}} \left| \Re\left( \frac{
		\Phi_{n}(u) - \Phi(u)
}{
		\Phi(u)
} 
\right)\right|\geq \frac{x}{n} \exp\left\{C_{1} \sigma^2 U_n^2 \int_{\mathbb{R}} (\mathcal{K}(x))^{2} dx \right\} 
\right\}.
\end{multline*}
Therefore, the choice \[x = K n \exp\left\{-C_{1} \sigma^2 U_n^2 \int_{\mathbb{R}} (\mathcal{K}(x))^{2} dx \right\} \eps_{n} /2= K \sqrt{n \log(n)}/2\] with any positive \(K\) leads to 
\begin{eqnarray*}
\P \left\{
	\max_{|u| \leq U_{n}}\left| \Re\left( \frac{
		\Phi_{n}(u) - \Phi(u)
}{
		\Phi(u)
} 
\right)\right|\geq 
\frac{
	K \eps_{n }
}{2} 
\right\}
\leq
\frac{
	\sqrt{2} c_{2}
}{
	\sqrt{K}
}
\frac{
	\sqrt{U_{n}} n^{(1/4) - c_{3}( K^{2}/4)}
}
{
	\log^{1/4}(n)
}.
\end{eqnarray*}
 Since the same statement holds for the imaginary bound of \(\left( \Phi_{n}(u) - \Phi(u) \right) / \Phi(u), \) we arrive at the desired result.




\section*{References}
\bibliographystyle{elsarticle-num}
\bibliography{bibliography-final-1}

\end{document}